\documentclass[11pt,a4paper]{article}

\usepackage[utf8]{inputenc}
\usepackage[T1]{fontenc}
\usepackage[english]{babel} 
\usepackage{amsmath}
\usepackage{amsfonts}
\usepackage{amsthm}
\usepackage{amssymb}
\usepackage{makeidx}
\usepackage{lmodern}
\usepackage[left=2cm,right=2cm,top=2cm,bottom=2cm]{geometry}

\usepackage{color}
\usepackage{comment}
\usepackage[dvipsnames]{xcolor}
\usepackage{graphicx}
\usepackage{framed}
\usepackage{calrsfs}
\DeclareMathAlphabet{\pazocal}{OMS}{zplm}{m}{n}
\usepackage{bm}

\usepackage{tikz}
\usetikzlibrary{shapes.misc}
\usetikzlibrary{decorations.pathreplacing}
\usetikzlibrary{patterns}
\usetikzlibrary{arrows,shapes,positioning}
\usetikzlibrary{shapes.multipart}
\usetikzlibrary{decorations.pathreplacing,calligraphy}

\usepackage{fourier-orns}
\usepackage{mathtools}
\usepackage{relsize}
\usepackage{dsfont}
\usepackage{comment}

\usepackage{hyperref}
\hypersetup{
	linktocpage = true,
	colorlinks = true,
	linkcolor = {RoyalBlue},
	urlcolor = {RoyalBlue},
	citecolor = {Green}}

\newcommand{\B}{\mathbb{B}}
\newcommand{\R}{\mathbb{R}}
\newcommand{\C}{\mathbb{C}}

\newcommand{\Apazo}{\pazocal{A}}
\newcommand{\Epazo}{\pazocal{E}}
\newcommand{\Fpazo}{\pazocal{F}}

\newcommand{\Upazo}{\pazocal{U}}
\newcommand{\Kpazo}{\pazocal{K}}

\newcommand{\Qpazo}{\pazocal{Q}}

\newcommand{\Lpazo}{\pazocal{L}}
\newcommand{\Ppazo}{\pazocal{P}}
\newcommand{\Ypazo}{\pazocal{Y}}

\newcommand{\Dpazo}{\pazocal{D}}

\newcommand{\Opazo}{\pazocal{O}}

\newcommand{\Xpazo}{\pazocal{X}}

\newcommand{\Mcal}{\mathcal{M}}
\newcommand{\Lcal}{\mathcal{L}}

\newcommand{\supp}{\textnormal{supp}}
\newcommand{\dom}{\textnormal{dom}}

\newcommand{\Lip}{\textnormal{Lip}}

\newcommand{\Graph}{\textnormal{Graph}}

\newcommand{\Div}{\textnormal{div}}

\newcommand{\dist}{\textnormal{dist}}
\newcommand{\textbn}[1]{\textnormal{\textbf{#1}}}
\newcommand{\co}{\overline{\textnormal{co}} \hspace{0.05cm}}

\newcommand{\dsf}{\textnormal{\textsf{d}}}

\newcommand{\mf}{\mathfrak{m}}

\newcommand{\Bnu}{\boldsymbol{\nu}}

\newcommand{\BPhi}{\boldsymbol{\Phi}}

\newcommand{\INTDom}[3]{\int_{#2} #1 \textnormal{d} #3}
\newcommand{\INTSeg}[4]{\int_{#3}^{#4} #1 \textnormal{d} #2}

\newcommand{\NormC}[3]{\left\| #1  \right\| _ {C^{#2}(#3)}}
\newcommand{\Norm}[1]{\parallel \hspace{-0.1cm} #1 \hspace{-0.1cm} \parallel}

\newcommand{\tderv}[2]{\tfrac{\textnormal{d} #1}{ \textnormal{d} #2}}

\newcommand{\adm}{\textnormal{adm}}

\newcommand{\Liminf}[1]{\underset{~ #1}{\textnormal{Liminf}}}
\newcommand{\Limsup}[1]{\underset{~ #1}{\textnormal{Limsup}}}
\newcommand{\Lim}[1]{\underset{~ #1}{\textnormal{Lim}}}
\newcommand{\tto}{\rightrightarrows}

\newtheorem{definition}{Definition}[section]
\newtheorem{theorem}[definition]{Theorem}
\newtheorem{proposition}[definition]{Proposition}
\newtheorem{remark}[definition]{Remark}

\newtheorem{corollary}[definition]{Corollary}
\newtheorem{example}[definition]{Example}

\numberwithin{figure}{section}
\numberwithin{equation}{section}
\renewcommand{\epsilon}{\varepsilon}

\newtheorem{taggedhypx}{Hypotheses}
\newenvironment{taggedhyp}[1]
    {\renewcommand\thetaggedhypx{#1}\taggedhypx}
    {\endtaggedhypx}
    
\newtheorem{taggedhypsingx}{Hypothesis}
\newenvironment{taggedhypsing}[1]
    {\renewcommand\thetaggedhypsingx{#1}\taggedhypsingx}
    {\endtaggedhypsingx}

\definecolor{dkgreen}{rgb}{0,0.4,0}

 \title{Set-Valued Koopman Theory for Control Systems}

\author{Benoît Bonnet-Weill\footnote{LAAS-CNRS, Université de Toulouse, CNRS, 7 avenue du colonel Roche, F-31400 Toulouse, France. Emails: \texttt{benoit.bonnet@laas.fr}, \texttt{korda@laas.fr}}  \; and Milan Korda$^{*,}$\footnote{Department of Control Engineering, Faculty of Electrical Engineering,
Czech Technical University in Prague, The Czech Republic}
}

\begin{document}

\maketitle

\begin{abstract}
In this paper, we introduce a new notion of Koopman operator which faithfully encodes the dynamics of controlled systems by leveraging the tools of set-valued analysis. In this context, we propose generalisations of the Liouville and Perron-Frobenius operators, and show that they respectively coincide with proper set-valued analogues of the infinitesimal generator and dual operator of the Koopman semigroup. We also give meaning to the spectra of these set-valued maps and prove an adapted version of the classical spectral mapping theorem relating the eigenvalues of a semigroup with those of its generator. Our approach provides theoretical justifications for existing practical methods in the Koopman community that study control systems by bundling together the Koopman and Liouville operators associated with different control inputs.  
\end{abstract}

{\footnotesize

\vspace{0.25cm}

\textbf{MSC2020 Subject Classification :} 28B20, 34A60, 47N20, 93C15.
}



\section{Introduction}

The Koopman operator was introduced in functional analysis as a way to provide an equivalent representation of nonlinear time-evolutions -- or more generally of flow maps -- in terms of infinite-dimensional linear operators. Originating from the seminal papers of Koopman \cite{Koopman1931} and Koopman and von Neumann \cite{koopman1932dynamical} in the early 1930s, it has enjoyed a renewed interest pioneered by the works of Mezi{\'c} and Banaszuk \cite{mezic2004comparison} and Mezi{\'c} \cite{Mezic2005}, and is now a well-established framework for the analysis of dynamical systems. Heuristically, given a one-parameter semigroup $(\Phi_t)_{t \geq 0}$ representing e.g. the integral curves of a dynamical system, the \textit{Koopman operators} are linear transforms whose action is given by 
\begin{equation}
\label{eq:IntroKoopman}
\Kpazo_t(\varphi) := \varphi \circ \Phi_t
\end{equation}
for each (typically real or complex valued) function $\varphi \in \Xpazo$ belonging to some relevant space of observables. Besides its appealing theoretical properties, which grant access to the whole corpus of spectral theory of linear operators to understand properties of nonlinear systems, Koopman operators have served as a foundation for state of the art numerical methods in computational mathematics, most notably the (Extended) Dynamic Mode Decomposition \cite{schmid2010dynamic,williams2015data}. We point the reader towards the books and surveys \cite{bevanda2021koopman,brunton2021modern,budivsic2012applied,kutz2016dynamic,mauroy2020koopman} for an overview of the many applications of this theory.

More recently, the Koopman framework has been used to study systems with external inputs, first regarded as exogenous disturbances \cite{proctor2018generalizing} that one cannot manipulate, and later, starting with the work \cite{korda2018linear}, as controls that one may tune in order to achieve specific goals. When used within a model predictive control framework, the approach proposed in the latter work took the form of a convex optimization problem, contrarily to traditional nonlinear model predictive schemes. This very appealing feature has led to a number of follow-up works aiming at improving the practical aspects of the method, see e.g. \cite{cibulka2022dictionary,korda2020optimal,peitz2019koopman,peitz2020data,shi2022deep}, as well as a number of applications ranging from soft robotics \cite{haggerty2023control} and power grid stabilisation \cite{korda2018power} to control designs for fluids \cite{arbabi2018data,peitz2020data} and quantum systems \cite{goldschmidt2022model}. This list is by no means exhaustive, and we point the interested reader to the surveys \cite{bevanda2021koopman,brunton2021modern} for more references. However, contrary to the abundance of methodological advances and applications of the Koopman framework, a sound theoretical footing for Koopman operators associated with controlled dynamics is still missing at present. Indeed, the work \cite{korda2018linear} defined the Koopman operator with control on the so-called tensor-product system, but did not leverage this definition for theoretical analysis. In \cite{proctor2018generalizing}, the authors considered the Koopman operator with one fixed value of the control input or with control signals subject to a dynamical evolution (e.g., determined by a feedback law), again without providing theoretical insights, while the subsequent work~\cite{peitz2020data} focused on the Koopman operators associated with each individual control input, and interpolated between them. The work~\cite{rosenfeld2021dynamic} likewise considered control-affine systems and defined the corresponding Koopman operator as the collection of the Koopman operators corresponding to the drift and each of the control vector fields.

In this article, we develop a new theoretical framework for Koopman operators in the presence of controls, based on the theory of set-valued analysis. Starting from the seminal works of Filippov \cite{Filippov1962} and Wazewski \cite{Wazewski1961} at the turn of the 1960s, it has been known that rephrasing control problems in terms of differential inclusions provides key insights on the optimal sets of assumptions needed to establish positive results on controlled dynamics, while bringing in powerful tools from geometric and nonsmooth analysis to investigate such systems. Since then, the methods of set-valued analysis have been successfully applied to a large breadth of control problems, ranging from the well-posedness of constrained dynamical systems both in the classical \cite{Bebernes1970} and hybrid \cite{Goebel2006} settings, to Pontryagin \cite{Clarke1976,Vinter1988} and Hamilton-Jacobi \cite{Frankowska1989,Frankowska1995} optimality conditions as well as Lyapunov stability methods \cite{Sontag1995}. This list of references is far from complete, and we point the interested reader to the monographs \cite{Aubin1984,Aubin1990,Clarke,Vinter} for further details. Based on these observations, we propose a comprehensive adaptation of the main concepts of Koopman theory to time-invariant control systems of the form
\begin{equation*}
\dot x(t) = f(x(t),u(t)), 
\end{equation*}
where $f : \R^d \times U \to \R^d$ is locally Lipschitz and sublinear in $x \in \R^d$ as well as continuous in $u \in U$. In this context, one can associate to each admissible control signal $u(\cdot) \in \Upazo$ a unique flow map $(\Phi_{(\tau,t)}^u)_{\tau,t \geq 0} \subset C^0(\R^d,\R^d)$ solving 
\begin{equation*}
\Phi_{(\tau,t)}^u(x) = x + \INTSeg{f\Big( \Phi_{(\tau,s)}^u(x) , u(s)\Big)}{s}{\tau}{t}
\end{equation*}
for all times $\tau,t \geq 0$ and every $x \in \R^d$. This leads us to defining the \textit{set-valued Koopman operators} as the collection of evaluations of a given observable $\varphi \in \Xpazo$ along all possible controlled flows, that is
\begin{equation*}
\Kpazo_{(\tau,t)}(\varphi) := \Big\{ \varphi \circ \Phi_{(\tau,t)}^u ~\, \textnormal{s.t.}~ u(\cdot) \in \Upazo  \Big\}.
\end{equation*}
In other words, $\Kpazo_{(\tau,t)}(\varphi)$ is the set of all observables which are reachable from $\varphi \in \Xpazo$ by right compositions with admissible controlled trajectories, see e.g. Figure \ref{fig:KoopmanDef} below. We discuss some of the fundamental topological properties of these objects in Section \ref{section:Koopman}, and provide a general representation formula for arbitrary time-dependent Koopman observables $(\tau,t) \in [0,T] \times [0,T] \mapsto \psi_{(\tau,t)} \in \Kpazo_{(\tau,t)}(\varphi)$, involving measurable families of control signals. It should be noted that, while autonomous controlled systems can be described up to a time reparametrisation by a one-parameter semigroup, we chose to define the Koopman operators as a two-parameter family in the spirit of \cite{Macesic2018}. The main incentive for doing so lies in the fact that the dynamics of Koopman observables depends on the starting time of the controlled flows, as shown in the next paragraph and more thoroughly in Section \ref{subsection:Liouville} below.

\medskip

\begin{figure}[!ht]
\centering
\resizebox{0.925\textwidth}{!}{
\begin{tikzpicture}
\draw[->] (-0.25,0)--(2.75,0);
\draw[->] (0,-0.25)--(0,1.75);
\draw[->] (0.15,0.15)--(-1,-1);
\draw (-0.5,-0.85) node {\scriptsize $\Xpazo$};
\draw (-0.25,1.5) node {\scriptsize $\C$};
\draw[gray] plot [smooth, tension=0.6] coordinates {(0.25,1)(0.75,0.95)(1.25,0.5)(1.75,0.515)(2.1,0.625)};
\draw[black, dashed] (0.25,-0.45)--(0.25,1);
\draw[black] (0.25,0.975) node {\LARGE $\cdot$};
\draw[black] (0.6,1.25) node {\scriptsize $\varphi(x)$};
\draw[black, dashed] (1.75,-0.725)--(1.75,0.5);
\draw (1.75,-0.735) node {\LARGE $\cdot$};
\draw (2.15,-1) node {\scriptsize $\Phi_t(x)$};
\draw[black] (1.75,0.495) node {\LARGE $\cdot$};
\draw[black] (2.65,0.85) node {\scriptsize $\varphi \circ \Phi_t(x)$};
\draw[black] plot [smooth, tension=0.8] coordinates {(0.25,-0.435)(0.75,-0.35)(1.45,-0.75)(2.25,-0.5)};
\draw (0.25,-0.45) node {\LARGE $\cdot$};
\draw (0.1,-0.6) node {\scriptsize $x$};
\begin{scope}[xshift=6cm]
\draw[->] (-0.25,0)--(2.75,0);
\draw[->] (0,-0.25)--(0,1.75);
\draw[->] (0.15,0.15)--(-1,-1);
\draw (-0.5,-0.85) node {\scriptsize $\Xpazo$};
\draw (-0.25,1.5) node {\scriptsize $\C$};
\filldraw[opacity = 0.8, draw = black, fill = white, thin, bottom color = white, top color = gray!40] plot [smooth, tension=0.6] coordinates {(2.5,-0.4)(2.25,-0.35)(1.45,-0.5)(0.75,-0.35)(0.25,-0.435)(0.75,-1)(1.45,-1.25)(2.25,-1.1)};
\draw[black, dashed, line width =0.05] plot [smooth, tension=0.8] coordinates {(0.25,-0.435)(0.75,-0.45)(1.45,-0.65)(2.25,-0.5)(2.45,-0.55)};
\draw[black, dashed, line width =0.05] plot [smooth, tension=0.8] coordinates {(0.25,-0.435)(0.75,-0.6)(1.45,-0.85)(2.25,-0.7)(2.35,-0.725)};
\draw[black, dashed, line width =0.05] plot [smooth, tension=0.8] coordinates {(0.25,-0.435)(0.75,-0.8)(1.45,-1.05)(2.3,-0.9)};
\draw[black, dashed, line width =0.4] (1.45,-0.5)--(1.45,0.65);
\draw[black, dashed, line width =0.4] (1.2,-1)--(1.2,0.35);
\draw[black, dashed, line width =0.4] (1.35,-0.85)--(1.35,0.5);
\draw[black, dashed, line width =0.4] (1,-1.15)--(1,0.35);
\draw[black, line width = 0.7] plot [smooth, tension=0.8] coordinates {(1.45,-0.5)(1.4,-0.85)(1,-1.15)};
\filldraw[opacity = 0.5, draw = gray, fill = white, thin, bottom color = white, top color = gray!40] plot [smooth, tension=0.8] coordinates {(2.5,0.7)(2.25,0.8)(1.455,0.7)(0.75,0.85)(0.25,1)(0.98,0.35)(1.8,0.15)(2.25,0.2)};
\draw[black, line width = 0.7] plot [smooth, tension=1] coordinates {(1.455,0.7)(1.45,0.55)(0.98,0.35)};
\draw (0.25,-0.455) node {\LARGE $\cdot$};
\draw (0.1,-0.6) node {\scriptsize $x$};
\draw (1.9,-1.6) node {\scriptsize $\{ \Phi_{(0,t)}^u(x)\}_{u(\cdot) \in \Upazo}$};
\draw[black] (2.85,1.1) node {\scriptsize $\{ \varphi \circ \Phi_{(0,t)}^u(x)\}_{u(\cdot) \in \Upazo}$};
\draw[black, dashed] (0.25,-0.45)--(0.25,1);
\draw[black] (0.25,1) node {\LARGE $\cdot$};
\draw[black] (0.65,1.2) node {\scriptsize $\varphi(x)$};
\end{scope}
\end{tikzpicture}
}
\vspace{-0.30cm}
\caption{{\small \textit{The classical Koopman operator associates to an observable $\varphi : \R^d \to \C$ the measurements $\varphi \circ \Phi_t(x)$ along a single flow $\Phi_t(x)$ starting from every possible initial data $x \in \R^d$ (left), whereas the set-valued Koopman operator outputs the measurements of the whole reachable set $\{ \varphi \circ \Phi_{(0,t)}^u(x)\}_{u(\cdot) \in \Upazo}$ along the collection of controlled flows $\{\Phi_{(0,t)}^u(x)\}_{u(\cdot) \in \Upazo}$ starting from any $x \in \R^d$ (right).}}}
\label{fig:KoopmanDef}
\end{figure} 

In Section \ref{section:LiouvillePerron}, which is the core of the manuscript, we put forth relevant set-valued counterparts of the \textit{Liouville} and \textit{Perron-Frobenius} operators. These objects are known to play a pivotal role in Koopman theory as well as many of its applications, the former by being the infinitesimal generator of the Koopman semigroup, and the latter by being the adjoint of the Koopman operators. In Section \ref{subsection:Liouville}, we define the \textit{set-valued Liouville operators} as 
\begin{equation*}
\Lpazo(\varphi) = \Big\{ \nabla_x \varphi \cdot f_u ~\, \mathrm{s.t.}~ u \in U \Big\}
\end{equation*}
for each continuously differentiable observable $\varphi \in \Dpazo$, where $f_u \in C^0(\R^d,\R^d)$ is given by $f_u(x) := f(x,u)$ for all $(x,u) \in \R^d \times U$. We then prove that, whenever the set of controlled vector fields is convex, the Liouville operator is the infinitesimal (set-valued) generator of the Koopman semigroup
\begin{equation*}
\Lim{t \to \tau} \frac{\Kpazo_{(\tau,t)}(\varphi)-\varphi}{t-\tau} = \Lpazo(\varphi), 
\end{equation*}
where the limit is understood in the sense of Kuratowski-Painlevé. In addition, when $f : \R^d \times U \to \R^d$ is continuously differentiable in $x \in \R^d$, we prove that time-dependent Koopman observables given by $(\tau,t) \in [0,T] \times [0,T] \mapsto \psi_{(\tau,t)} := \varphi \circ \Phi_{(\tau,t)}^u \in \Dpazo$ for some fixed signal $u(\cdot) \in \Upazo$ coincide precisely with the strong solutions of the differential inclusion 
\begin{equation*}
\partial_{\tau} \psi_{(\tau,t)} \in -\Lpazo(\psi_{(\tau,t)})
\end{equation*}
in the space of observables $\Dpazo$. In Section \ref{subsection:Perron}, we shift our focus to the investigation of duality results for the set-valued Koopman operators. In this context, we propose the following definition for the \textit{set-valued Perron-Frobenius} operators 
\begin{equation*}
\Ppazo_{(\tau,t)}(\mu) := \Big\{ \Phi_{(\tau,t) \sharp \,}^u \mu ~\, \text{s.t.}~ u(\cdot) \in \Upazo \Big\}, 
\end{equation*}
where ``$\sharp$'' stands for the usual image measure operation. We then show that the weak-$^*$ convex hull of the latter coincide with the set-valued adjoint of the Koopman operators, defined in the sense of Ioffe \cite{Ioffe1981}. We also leverage this notion of duality to prove that, again when the dynamics is convex, the infinitesimal generator of the Perron-Frobenius semigroup is the adjoint of the Liouville operator, namely 
\begin{equation*}
\Lim{t \to \tau} \frac{\Ppazo_{(\tau,t)}(\mu)-\mu}{t-\tau} = \Lpazo^*(\mu)
\end{equation*}
in the weak-$^*$ topology. To our surprise, we discovered that the underlying dynamics
\begin{equation*}
\partial_t \mu_{(\tau,t)} \in \Lpazo^*(\mu_{(\tau,t)})
\end{equation*}
coincided -- at least formally -- with the notion of \textit{continuity inclusion} introduced by the first author and Frankowska in \cite{ContInc,ContIncPp} in the context of meanfield control. Lastly, in Section \ref{subsection:Spectral}, we prove a set-valued version of the usual spectral mapping theorem (see e.g. \cite[Chapter IV - Theorem 3.7]{Engel2001}), which relates the point spectra of the Liouville operator and Koopman semigroup. Therein, we show in particular that 
\begin{equation*}
e^{(t-\tau) \sigma_p(\Lpazo)} \subset \sigma_p(\Kpazo_{(\tau,t)}), 
\end{equation*}
which provides grounding to pre-existing works such as  \cite{peitz2019koopman,peitz2020data} in which spectral properties of Koopman operators for control systems are investigated by exponentiating the spectra of the Liouville operators associated with a finite collection of controls.  

The manuscript is organised as follows. In Section \ref{section:Preliminaries}, we start by exposing preliminary notions of functional and set-valued analysis as well as control theory. In Section \ref{section:Koopman}, we define the set-valued Koopman operators and study their main properties. We then move to the investigation of the set-valued Liouville and Perron-Frobenius operators in Section \ref{section:LiouvillePerron}, wherein we show that the latter are respectively related to the infinitesimal generator and adjoint of the Koopman semigroup. We then close this section with the statement of a set-valued counterpart of spectral mapping theorem, and provide the proof of a regularity result for controlled flows in Appendix \ref{section:AppendixContinuity}.  


\section{Preliminaries}
\label{section:Preliminaries}
\setcounter{equation}{0} \renewcommand{\theequation}{\thesection.\arabic{equation}}

In this first section, we collect preliminary material on integration theory, set-valued analysis and controlled systems, for which we point the reader to the monographs \cite{AnalysisBanachSpaces}, \cite{Aubin1990} and \cite{Clarke} respectively. 


\subsection{Functional analysis and integration}

In what follows, we denote by $(\Omega,\Apazo,\mf)$ a complete $\sigma$-finite measure space, and recall following \cite[Section 8.1]{Aubin1990} that a map $f : \Omega \to \Xpazo$ valued in a Polish space $(\Xpazo,\dsf_{\Xpazo}(\cdot,\cdot))$ is \textit{$\mf$-measurable} if the set
\begin{equation*}
f^{-1}(\Opazo) := \Big\{ \omega \in \Omega ~\, \textnormal{s.t.}~ f(\omega) \in \Opazo \Big\} \subset \Omega
\end{equation*} 
is $\mf$-measurable for each open set $\Opazo \subset \Xpazo$. Throughout the manuscript, we shall denote by $\Lcal^d$ the standard Lebesgue outer measure defined over $\R^d$. Below, we recall the concept of integrability in the sense of Bochner for maps valued in separable Banach spaces, see for instance \cite[Chapter 1]{AnalysisBanachSpaces}. 

\begin{definition}[Bochner integrable maps]
An $\mf$-measurable map $f : \Omega \to \Xpazo$ valued in a separable Banach space $ (\Xpazo,\|\cdot\|_{\Xpazo})$ is said to be \textnormal{Bochner integrable} if 
\begin{equation*}
\INTDom{\Norm{f(\omega)}_{\Xpazo}}{\Omega}{\mf(\omega)} < +\infty.
\end{equation*}
The collection of all such maps is a separable Banach space denoted by $L^1(\Omega,\Xpazo)$. 
\end{definition}

%
%

Given a locally convex topological vector space $\Epazo$, we denote by $\Epazo^*$ its topological dual, i.e. the collection of all bounded complex-valued linear functionals over $\Epazo$, and write $\langle \cdot,\cdot\rangle_{\Epazo}$ for the underlying duality pairing, see e.g. \cite[Part 1 -- Chapter 1]{Rudin1991}. Throughout the paper, we always assume that a dual space is endowed with the weak-$^*$ topology, see e.g. \cite[Part 1 -- Chapter 3]{Rudin1991}. In the next proposition, we recall a fine weak compactness criterion in $L^1(\Omega,\Xpazo)$ excerpted from \cite[Corollary 2.6]{Diestel1993}.

\begin{proposition}[A weak $L^1$-compactness criterion for the Bochner integral]
\label{prop:WeakCompactness}
Let $(\Xpazo,\Norm{\cdot}_{\Xpazo})$ be a separable Banach space and $(f_n(\cdot)) \subset L^1(\Omega,\Xpazo)$. Suppose that there exists a map $k(\cdot) \in L^1(\Omega,\R_+)$ along with a convex and compact set $K \subset \Xpazo$ such that 
\begin{equation*}
\Norm{f_n(\omega)}_{\Xpazo} \, \leq k(\omega) \qquad \text{and} \qquad f_n(\omega) \in K    
\end{equation*}
for $\mf$-almost every $\omega \in \Omega$. Then, there exists a subsequence $(f_{n_k}(\cdot))$ that converges weakly to some $f(\cdot) \in L^1(\Omega,\Xpazo)$ which satisfies $f(\omega) \in K$ for $\mf$-almost every $\omega \in \Omega$. In particular, it holds that
\begin{equation*}
\INTDom{\big\langle \Bnu(\omega) , f(\omega) - f_{n_k}(\omega) \big\rangle_{\Xpazo} \,}{\Omega}{\mf(\omega)} ~\underset{k \to + \infty}{\longrightarrow}~ 0    
\end{equation*}
for each map $\Bnu : \Omega \to \Xpazo^*$ such that $\omega \in \Omega \mapsto \langle \Bnu(\omega) , f \, \rangle_{\Xpazo} \in \R$ is $\mf$-measurable whenever $f \in \Xpazo$ and satisfying $\textnormal{ess-sup}_{\omega \in \Omega} \Norm{\Bnu(\omega)}_{\Xpazo^*} < +\infty$.
\end{proposition}
 
Given $k \in \{0,1\}$, we denote by $(C^k_c(\R^d,\C),\NormC{\cdot}{k}{\R^d,\C})$ the separable normed space of $k$-times continuously differentiable maps with compact support endowed with the relevant supremum norm, and use the notation $(C^k_0(\R^d,\C),\NormC{\cdot}{k}{\R^d,\C})$ for its norm completion. By Riesz's representation theorem (see e.g. \cite[Theorem 1.54]{AmbrosioFuscoPallara}), one has that $C_c^0(\R^d,\C)^* \simeq \Mcal(\R^d,\C)$ where $\Mcal(\R^d,\C)$ is the vector space of finite complex-valued Radon measures. In this context, given an element $\mu \in \Mcal(\R^d,\C)$ and some $f \in C^0(\R^d,\R^d)$, we define the so-called \textit{divergence distribution} $\Div_x(f \mu) \in C^1_c(\R^d,\C)^*$ by duality as 
\begin{equation}
\label{eq:DivergenceDef}
\langle \Div_x(f \mu) , \zeta\rangle_{C^1_c(\R^d,\C)} := -\INTDom{\nabla_x \zeta(x) \cdot f(x)}{\R^d}{\mu(x)}
\end{equation}
for all $\zeta \in C^1_c(\R^d,\C)$. We also recall that the \textit{image} -- or \textit{pushforward} -- of a finite Radon measure $\mu \in \Mcal(\R^d,\C)$ through a Borel map $f : \R^d \to \R^d$ is the unique measure which satisfies $f_{\sharp} \mu(B) := \mu(f^{-1}(B))$ for each Borel set $B \subset \R^d$. Besides, the latter is characterised by the change of variable formula
\begin{equation}
\label{eq:ImageMeasure}
\INTDom{\varphi(x)}{\R^d}{(f_{\sharp}\mu)(x)} = \INTDom{\varphi \circ f(x)}{\R^d}{\mu(x)} 
\end{equation}
for every Borel map $\varphi : \R^d \to [0,+\infty]$. In the next definition, we recollect the notion of \textit{weak-$^*$ convergence} for finite Radon measures (see e.g. \cite[Definition 1.58]{AmbrosioFuscoPallara}).

\begin{definition}[Weak-$^*$ convergence of measures]
A sequence $(\mu_n) \subset \Mcal(\R^d,\C)$ converges towards $\mu \in \Mcal(\R^d,\C)$ for the \textnormal{weak-$^*$ topology} provided that 
\begin{equation}
\label{eq:weakstar}
\INTDom{\zeta(x)}{\R^d}{\mu_n(x)} ~\underset{n \to +\infty}{\longrightarrow}~ \INTDom{\zeta(x)}{\R^d}{\mu(x)}
\end{equation}
for each $\zeta \in C^0_c(\R^d,\C)$. 
\end{definition}

In what ensues, we will also consider elements of the space $C^0(\R^d,\R^d)$ of continuous functions defined over the whole of $\R^d$. In that case, the adequate topology to consider is that of \textit{local uniform convergence}, whose definition is recalled below. Therein, we write $B(0,R) \subset \R^d$ for the standard closed Euclidean ball of radius $R > 0$.

\begin{definition}[Topology of local uniform convergence]
A sequence of maps $(f_n) \subset C^0(\R^d,\R^d)$ converges \textnormal{locally uniformly} towards $f \in C^0(\R^d,\R^d)$ provided that
\begin{equation*}
\NormC{f - f_n}{0}{K,\R^d} ~\underset{n \to +\infty}{\longrightarrow}~ 0    
\end{equation*}
for each compact set $K \subset \R^d$. The underlying topology is induced by the translation invariant metric 
\begin{equation*}
\dsf_{cc}(f,g) := \sum_{k =1}^{+\infty} 2^{-k} \min \Big\{ 1 \, , \, \NormC{f-g}{0}{B(0,k),\R^d} \hspace{-0.05cm} \Big\}, 
\end{equation*}
which is defined for each $f,g \in C^0(\R^d,\R^d)$.
\end{definition}

Following the general results from \cite[Chapter 7 -- Theorems 12 and 13]{Kelley1975} and \cite[Theorem 6]{Warner1958}, it can be shown $(C^0(\R^d,\R^d),\dsf_{cc}(\cdot,\cdot))$ is a \textit{Fréchet space}, i.e. a complete separable locally convex topological vector space whose topology is induced by a translation invariant metric. More classically, it is known that $(C^0(K,\R^d),\NormC{\cdot}{0}{K,\R^d})$ is a separable Banach space for each compact set $K \subset \R^d$.


\subsection{Set-valued analysis}
\label{subsection:Setvalued}

In this section, we recall preliminary material pertaining to set-valued analysis, for which we point the reader to the reference monograph \cite{Aubin1990}. In what follows, unless further specifications are given, we suppose that $\Xpazo$ and $\Ypazo$ are separable locally convex topological vector spaces.

\begin{definition}[Set-valued maps]
We say that $\Fpazo : \Xpazo \tto \Ypazo$ is a \textnormal{set-valued map} if $\Fpazo(x) \subset \Ypazo$ for each $x \in \Xpazo$. Its \textit{domain} and \textit{graph} are defined respectively by
\begin{equation*}
\dom(\Fpazo) := \Big\{ x \in \Xpazo ~\, \textnormal{s.t.}~ \Fpazo(x) \neq \emptyset \Big\}, ~~ \Graph(\Fpazo) := \Big\{ (x,y) \in \Xpazo \times \Ypazo ~\,\textnormal{s.t.}~ y \in \Fpazo(x) \Big\}.
\end{equation*}
\end{definition}

Below, we recollect classical regularity notions for set-valued mappings, starting with the standard concepts of continuity and Lipschitz continuity, and proceed with that of measurability. Therein, we suppose that $(\Xpazo,\dsf_{\Xpazo}(\cdot,\cdot))$ and $(\Ypazo,\dsf_{\Ypazo}(\cdot,\cdot))$ are Polish spaces, write $\B_{\Xpazo}(x,r)$ for the closed ball of radius $r >0$ centred at some $x \in \Xpazo$ and let $\B_{\Xpazo}(\Qpazo,r) := \{y \in \Xpazo ~\, \text{s.t.}~ \dsf_{\Xpazo}(x,y) \leq r ~~ \text{for some $x \in \Qpazo$}\}$ for any $\Qpazo \subset \Xpazo$.

\begin{definition}[Continuity of set-valued maps]
\label{def:Continuity}
A set-valued map $\Fpazo : \Xpazo \tto \Ypazo$ is said to be \textnormal{continuous} at $x \in \dom(\Fpazo)$ if  both the following conditions hold. 
\begin{enumerate}
\item[$(i)$] $\Fpazo$ is \textnormal{lower-semicontinuous}, i.e. for any $\epsilon > 0$ and all $y \in \Fpazo(x)$, there exists $\delta > 0$ such that 
\begin{equation*}
\Fpazo(x') \cap \B_{\Ypazo}(y,\epsilon) \neq \emptyset
\end{equation*}
for each $x' \in \B_{\Xpazo}(x,\delta)$.
\item[$(ii)$] $\Fpazo$ is \textnormal{upper-semicontinuous}, i.e. for any $\epsilon > 0$, there exists $\delta >0$ such that 
\begin{equation*}
\Fpazo(x') \subset \B_{\Ypazo}(\Fpazo(x),\epsilon)
\end{equation*}
for each $x' \in \B_{\Xpazo}(x,\delta)$.
\end{enumerate} 
\end{definition}

\begin{definition}[Lipschitz continuity of set-valued maps]
A set-valued mapping $\Fpazo :\Xpazo \tto \Ypazo$ is \textnormal{Lipschitz continuous} with constant $L > 0$ provided that 
\begin{equation*}
\Fpazo(x') \subset \B_{\Ypazo} \Big(\Fpazo(x) , L \, \dsf_{\Xpazo}(x,x') \Big)
\end{equation*}
for all $x,x' \in \dom(\Fpazo)$. In particular, $\Fpazo$ is continuous at every point $x \in \dom(\Fpazo)$.
\end{definition}

\begin{definition}[Measurable set-valued maps and selections]
A set-valued map $\Fpazo : \Omega \tto \Xpazo$ is said to be \textnormal{$\mf$-measurable} -- or more simply measurable -- if for any open set $\Opazo \subset \Xpazo$, the preimage
\begin{equation*}
\Fpazo^{-1}(\Opazo) := \Big\{ \omega \in \Omega ~\, \textnormal{s.t.}~ \Fpazo(\omega) \cap \Opazo \neq \emptyset \Big\} \subset \Omega
\end{equation*}
is $\mf$-measurable. Moreover, a measurable map $f : \Omega \to \Xpazo$ is called a \textnormal{measurable selection} of $\Fpazo : \Omega \tto \Xpazo$ provided that $f(\omega) \in \Fpazo(\omega)$ for $\mf$-almost every $\omega \in \Omega$.
\end{definition}

One may note that whenever $\Fpazo(\omega) = \{f(\omega)\}$ is single-valued, this definition amounts to requiring that $f : \Omega \to \Xpazo$ be $\mf$-measurable, since then
\begin{equation*}
\Fpazo^{-1}(\Opazo) = \Big\{ \omega \in \Omega ~\, \textnormal{s.t.}~ \{f(\omega)\} \cap \Opazo \neq \emptyset \Big\} = f^{-1}(\Opazo).    
\end{equation*}
In the following theorem, we recall a variant of the fundamental Filippov selection principle, whose statement may be found e.g. in \cite[Theorem 8.2.10]{Aubin1990}. 

\begin{theorem}[Filippov's measurable selection principle]
\label{thm:FilippovSel}
Let $\Fpazo : \Omega \tto \Xpazo$ be an $\mf$-measurable set-valued map with nonempty closed images and $\Psi : \Omega \times \Xpazo \to \Ypazo$ be a mapping that is $\mf$-measurable in $\omega \in \Omega$ and continuous in $x \in \Xpazo$. Then, for every $\mf$-measurable map $\psi : \Omega \to \Ypazo$ satisfying
\begin{equation*}
\psi(\omega) \in \Psi(\omega,\Fpazo(\omega)) := \Big\{ \Psi(\omega,f) ~\, \textnormal{s.t.}~ f \in \Fpazo(\omega) \Big\},
\end{equation*}
there exists a measurable selection $\omega \in \Omega \mapsto f(\omega) \in \Fpazo(\omega)$ such that $\psi(\omega) = \Psi(\omega,f(\omega))$ for $\mf$-almost every $\omega \in \Omega$.
\end{theorem}

In what follows, we recollect the definitions of lower and upper limits for sequences of sets. Therein, we shall use the notation  
\begin{equation*}
\dist_{\Xpazo}(x \, ; \Qpazo) := \inf_{y \in \Qpazo} \dsf_{\Xpazo}(x,y)
\end{equation*}
for the distance between a point $x \in \Xpazo$ and an arbitrary set $\Qpazo \subset \Xpazo$. 

\begin{definition}[Kuratowski-Painlevé limit of sequences of sets]
\label{def:Kuratowski-Painlevé}
Given a sequence of sets $(K_n)$, we define its \textnormal{lower limit} as
\begin{equation*}
\Liminf{n \to +\infty} \, K_n := \bigg\{ x \in \Xpazo ~\,\textnormal{s.t.}~ \lim_{n \to +\infty} \, \dist_{\Xpazo}(x \, ; K_n) = 0 \bigg\}, 
\end{equation*}
as well as its \textnormal{upper limit} by
\begin{equation*}
\Limsup{n \to +\infty} \, K_n := \bigg\{ x \in \Xpazo ~\,\textnormal{s.t.}~ \liminf_{n \to +\infty} \, \dist_{\Xpazo}(x \, ; K_n) = 0 \bigg\}.
\end{equation*}
In this context, we say that $K \subset\Xpazo$ is the \textnormal{Kuratowski-Painlevé limit} of a sequence of sets $(K_n)$ if 
\begin{equation*}
K = \Lim{n \to +\infty} \, K_n := \Liminf{n \to +\infty} \, K_n = \Limsup{n \to +\infty} \, K_n.
\end{equation*}
\end{definition}

We close this section by recollecting the definition of closed processes and fans between locally convex topological vector spaces, see Figure \ref{fig:Fans} below for an illustration. The latter were introduced respectively by Rockafellar at the end of the 1960s, see e.g. \cite[Chapter 9]{Rockafellar2015}, and Ioffe in the 1980's, see e.g. \cite{Ioffe1981}, with the goal to adapt the main and most appealing properties of closed and linear operators to the set-valued setting.

\begin{definition}[Closed processes and fans]
\label{def:ClosedProcess}
A set-valued map $\Fpazo : \Xpazo \tto \Ypazo$ is called a \textnormal{closed process} if its graph is a closed cone. A closed process $\Fpazo : \Xpazo \tto \Ypazo$ with convex images is called a \textnormal{fan} provided that 
\begin{equation*}
0 \in \Fpazo(0) \qquad \text{and} \qquad \Fpazo(x_1 + x_2) \subset \overline{\Fpazo(x_1) + \Fpazo(x_2)} 
\end{equation*}
for each $x_1,x_2 \in \dom(\Fpazo)$, namely if it is positively homogeneous and subadditive. 
\end{definition}

\begin{example}[Collections of linear operators as closed processes and fans]
To see why these objects can indeed be seen as generalisations of bounded linear operators, suppose that $\Fpazo(x) := \{Ax\}$ for each $x \in \Xpazo$ and some bounded linear map $ A : \Xpazo \to \Ypazo$. Then, one may easily check that 
\begin{equation*}
\Graph(\Fpazo) = \big\{ (x,Ax) ~\, \textnormal{s.t.}~ x \in \Xpazo \big\}= \Graph(A)
\end{equation*}
is a linear space and that the images of $\Fpazo : \Xpazo \tto \Ypazo$ are convex, hence the latter defines a fan. In the case in which $A : \dom(A) \to \Ypazo$ is a linear operator defined over some domain $\dom(A) \subset \Xpazo$, the mapping $\Fpazo : \Xpazo \tto \Ypazo$ is also a fan. More interestingly, if $\Apazo$ is a convex and closed collection of linear operators from $\Xpazo$ to $\Ypazo$, then $\Fpazo(x) := \{ Ax ~\, \textnormal{s.t.}~ A \in \Apazo \,\}$ defines a fan.      
\end{example} 

\begin{figure}[!ht]
\centering
\resizebox{0.75\textwidth}{!}{
\begin{tikzpicture}
\begin{scope}
\draw[black, domain=0:2, line width = 0.4mm] plot (\x,{0.75*\x});
\draw[black, domain=0:2, line width = 0.4mm] plot (2-\x,{(-1.5+0.75*\x)}); 
\draw[black, domain=0:-2, line width = 0.4mm] plot (\x,{0.75*\x});
\draw[black, domain=0:-2, line width = 0.4mm] plot (-2-\x,{1.5+0.75*\x}); 
\fill[opacity = 1, thin, bottom color = white, top color = gray!40, domain = 0:2, variable=\x]  (0,0) -- plot (\x,{0.75*\x}) -- plot (2-\x,{-1.5+0.75*\x}) -- cycle;  
\fill[opacity = 1, thin, bottom color = white, top color = gray!40, domain = 0:-2, variable=\x]  (0,0) -- plot (\x,-0.75*\x) -- plot (-2-\x,{-(1.5+0.75*\x)}) -- cycle;  
\foreach \x in {-3.5,-2.5,-1.5,-0.5,0.5,1.5,2.5,3.5}
	\draw[dashed, line width = 0.01mm] (\x/2,-0.75*\x/2)--(\x/2,0.75*\x/2); 
\draw[->] (-2.25,0) -- (2.5,0); 
\draw[->] (0,-1.5) -- (0,1.5);
\draw (2.5,-0.3) node {\small $\Xpazo$};
\draw (0.35,1.5) node {\small $\Ypazo$};
\end{scope}
\begin{scope}[xshift=8cm]
\fill[opacity = 1, thin, bottom color = white, top color = gray!40, domain = 0:2, variable=\x]  (0,0) -- plot (\x,{0.75*\x}) -- plot (2-\x,{1-0.5*\x}) -- cycle;  
\fill[rotate = 180, opacity = 1, thin, bottom color = gray!40, top color = white, domain = 0:2, variable=\x]  (0,0) -- plot (\x,{0.75*\x}) -- plot (2-\x,{1-0.5*\x}) -- cycle;  
\fill[opacity = 1, thin, bottom color = white, top color = gray!40, domain = 0:2, variable=\x]  (0,0) -- plot (\x,{0.25*\x}) -- plot (2-\x,{-0.5+0.25*\x}) -- cycle;  
\fill[rotate = 180, opacity = 1, thin, bottom color = white, top color = gray!40, domain = 0:2, variable=\x]  (0,0) -- plot (\x,{0.25*\x}) -- plot (2-\x,{-0.5+0.25*\x}) -- cycle;  
\fill[opacity = 1, thin, top color = white, bottom color = gray!40, domain = 0:2, variable=\x]  (0,0) -- plot (\x,{-0.75*\x}) -- plot (2-\x,{-1+0.5*\x}) -- cycle; 
\fill[rotate = 180, opacity = 1, thin, bottom color = white, top color = gray!40, domain = 0:2, variable=\x]  (0,0) -- plot (\x,{-0.75*\x}) -- plot (2-\x,{-1+0.5*\x}) -- cycle; 
\draw[->] (-2.25,0) -- (2.5,0); 
\draw[->] (0,-1.5) -- (0,1.5);
\draw[black, domain=-2:2, line width = 0.01mm] plot (\x,{0.75*\x});
\draw[black, domain=-2:2, line width = 0.01mm] plot (\x,{0.5*\x});
\draw[black, domain=-2:2, line width = 0.01mm] plot (\x,{0.25*\x});
\draw[black, domain=-2:2, line width = 0.01mm] plot (\x,{-0.25*\x});
\draw[black, domain=-2:2, line width = 0.01mm] plot (\x,{-0.5*\x});
\draw[black, domain=-2:2, line width = 0.01mm] plot (\x,{-0.75*\x});
\foreach \x in {-3.5,-2.5,-1.5,-0.5,0.5,1.5,2.5,3.5}
	\draw[dashed, line width = 0.01mm] (\x/2,0.5*\x/2)--(\x/2,0.75*\x/2); 
\foreach \x in {-3.5,-2.5,-1.5,-0.5,0.5,1.5,2.5,3.5}
	\draw[dashed, line width = 0.01mm] (\x/2,-0.25*\x/2)--(\x/2,0.25*\x/2);
\foreach \x in {-3.5,-2.5,-1.5,-0.5,0.5,1.5,2.5,3.5}
	\draw[dashed, line width = 0.01mm] (\x/2,-0.5*\x/2)--(\x/2,-0.75*\x/2); 
\draw (2.5,-0.3) node {\small $\Xpazo$};
\draw (0.35,1.5) node {\small $\Ypazo$};
\end{scope}
\end{tikzpicture}
}
\caption{{\small \textit{Illustration of the concepts of closed processes and fans. The set-valued mapping given by $\Fpazo(x) := [-x,x]$ (left) is clearly homogeneous and convex-valued as well as subadditive, and thus defines a fan. On the other hand, the set-valued map given by $\Fpazo(x) = [-x,-3x/4] \cup [-x/2,x/2] \cup [3x/4,x] $ (right) is a homogeneous closed process, but has non-convex images and is therefore not a fan.}}}
\label{fig:Fans}
\end{figure} \vspace{-0.5cm}


\subsection{Controlled dynamical systems}

In what follows, we recollect elementary facts pertaining to controlled Cauchy problems of the form
\begin{equation}
\label{eq:ContCauchy}
\left\{
\begin{aligned}
\dot x(t) & = f(x(t),u(t)), \\
x(\tau) & = x,
\end{aligned}
\right.
\end{equation}
defined over some finite time interval $[0,T]$, and taking a prescribed value $x \in \R^d$ at some $\tau \in [0,T]$. Therein, the set of admissible controls is given by
\begin{equation*}
\Upazo := \Big\{ u : [0,T] \to U ~\, \text{s.t.}~ u(\cdot) ~\text{is $\Lcal^1$-measurable} \Big\}    
\end{equation*}
where $(U,\dsf_U(\cdot,\cdot))$ is a compact metric space. Throughout the manuscript, we suppose that the following assumptions hold.   

\begin{taggedhyp}{\textbn{(H)}} \hfill
\label{hyp:H}
\begin{enumerate}
\item[$(i)$] The map  $f: \R^d\times U \to \R^d$ is continuous, and there exists a constant $m > 0$ such that 
\begin{equation*}
|f(x,u)| \leq m \big(1 + |x| \big)
\end{equation*}
for all $(x,u) \in\R^d \times U$. 
\item[$(ii)$] For each compact set $K \subset\R^d$, there exists a constant $\ell_K >0$ such that 
\begin{equation*}
|f(x,u) - f(y,u)| \leq \ell_K |x-y|
\end{equation*}
for all $x,y \in K$ and each $u \in U$.
\end{enumerate}
\end{taggedhyp}

\begin{remark}[Concerning Hypotheses \textnormal{\ref{hyp:H}}]
Since we restrict our attention to compact control sets, our assumptions encompass the linear controlled dynamics 
\begin{equation*}
f(x,u) := Ax + Bu
\end{equation*}
where $U \subset \R^n$ is compact and $A \in \R^{d \times d}$, $B \in \R^{d \times n}$ are matrices. We also stress that all our results could be extended to unbounded time intervals e.g. by assuming that the dynamics is stable, in the sense that solutions starting from a compact set of initial data remain inside a (possibly larger) compact set.
\end{remark}

Throughout our developments, we will often use the notation $f_u \in C^0(\R^d,\R^d)$ for the vector field $x \in \R^d \mapsto f(x,u) \in\R^d$ associated with some $u \in U$, and denote by 
\begin{equation*}
\Fpazo := \Big\{ f_u \in C^0(\R^d,\R^d) ~\, \textnormal{s.t.}~ u \in U \Big\}
\end{equation*}
the set of all admissible fields of the control system. In the following proposition, we show that the latter is compact under our working assumptions. 

\begin{proposition}[Compactness of controlled vector fields]
\label{prop:Compact}
Under Hypotheses \textnormal{\ref{hyp:H}}, the set $\Fpazo \subset C^0(\R^d,\R^d)$ is compact for the topology of local uniform convergence.
\end{proposition}

\begin{proof}
Let $(f_n) \subset \Fpazo$ be a sequence of admissible vector fields. By construction, there exists a sequence of controls $(u_n) \subset U$ such that $f_n = f_{u_n}$ for each $n \geq 1$. Since $(U,\dsf_U(\cdot,\cdot))$ is compact, there exists an element $u \in U$ for which 
\begin{equation*}
\dsf_U(u_{n_k},u) ~\underset{n \to +\infty}{\longrightarrow}~ 0    
\end{equation*}
along a subsequence $(u_{n_k}) \subset U$. Thus, it follows from Hypothesis \textnormal{\ref{hyp:H}}-$(i)$ that 
\begin{equation*}
\big| f_u(x) - f_{u_{n_k}}(x) \big| ~\underset{n \to +\infty}{\longrightarrow}~ 0
\end{equation*} 
for all $x \in \R^d$. Besides, it stems from Hypothesis \textnormal{\ref{hyp:H}}-$(ii)$ that the maps $f_{u_n} : K \to \R^d$ are uniformly equicontinuous over each compact set $K \subset \R^d$, so that 
\begin{equation*}
\sup_{x \in K} \big| f_u(x) - f_{u_n}(x) \big| ~\underset{n \to +\infty}{\longrightarrow}~ 0.       
\end{equation*}
The latter identity being valid for all compact set, this concludes the proof. 
\end{proof}

In the following theorem, we recall  a standard Cauchy-Lipschitz well-posedness result for the controlled dynamics \eqref{eq:ContCauchy}, along with elementary estimates satisfied by the underlying flow maps. We point the reader to \cite[Chapter 2]{BressanPiccoli} for detailed proofs.

\begin{theorem}[Well-posedness, stability and representation of solutions]
\label{thm:Well-posed}
Let $(\tau,x) \in [0,T] \times \R^d $ be given and suppose that Hypotheses \textnormal{\ref{hyp:H}} hold. Then for each $u(\cdot) \in \Upazo$, the dynamics \eqref{eq:ContCauchy} admits a unique solution $x_u(\cdot) \in \Lip([0,T],\R^d)$ which can be represented explicitly as  
\begin{equation}
x_u(t) = \Phi_{(\tau,t)}^u(x)
\end{equation}
for all times $t \in[0,T]$. Therein, the family of maps $(\Phi_{(\tau,t)}^u)_{\tau,t \in [0,T]} \subset C^0(\R^d,\R^d)$ are the \textnormal{flows of homeomorphisms} defined as the unique solution of the Cauchy problem  
\begin{equation}
\label{eq:ContFlow}
\Phi_{(\tau,t)}^u(x) = x + \INTSeg{f \Big( \Phi_{(\tau,s)}^u(x) , u(s) \Big)}{s}{\tau}{t}.
\end{equation}
In addition, for each $R>0$, the latter comply with the estimates
\begin{equation}
\label{eq:StabEst}    
\big| \Phi_{(\tau_1,t_1)}^u(x) \big| \leq M_R \quad \text{and} \quad \big| \Phi_{(\tau_1,t_1)}^u(x) - \Phi_{(\tau_2,t_2)}^u(y) \big| \leq L_R \Big(|\tau_1 - \tau_2| + |t_1 - t_2| + |x - y| \Big),
\end{equation}
for all times $\tau_1,\tau_2,t_1,t_2 \in [0,T]$ and every $x,y \in B(0,R)$, where $M_R,L_R > 0$ are constants which only depend on the magnitudes of $m,T$ and $R$.  
\end{theorem}

We close this section by recollecting a folklore result stating that flow solutions of \eqref{eq:ContFlow} depend continuously on their control input. Therein and in what follows, we will use the observation that $\Upazo$ is a Polish space when endowed with the distance
\begin{equation*}
\dsf_{\Upazo}(u(\cdot),v(\cdot)) := \INTSeg{\dsf_U(u(t),v(t))}{t}{0}{T}
\end{equation*}
defined for all $u(\cdot),v(\cdot) \in \Upazo$, see for instance \cite[Chapter 8]{Aubin1990}.   

\begin{corollary}[Continuity in the control]
\label{cor:ContinuousFlow}
Under Hypotheses \textnormal{\ref{hyp:H}}, the map
\begin{equation*}
u(\cdot) \in \Upazo \mapsto \Phi_{(\cdot,\cdot)}^u(\cdot) \in C^0([0,T] \times [0,T] \times \R^d,\R^d)
\end{equation*}
which to a control signal associates the corresponding flow is continuous.
\end{corollary}

\begin{proof}
Being standard, the proof of this result is deferred to Appendix \ref{section:AppendixContinuity}.
\end{proof}


\section{Set-valued Koopman operators}
\label{section:Koopman}
\setcounter{equation}{0} \renewcommand{\theequation}{\thesection.\arabic{equation}}

In this section, we introduce the \textit{set-valued Koopman operators} associated with a control system, and discuss some of their fundamental properties. In what follows, we assume that Hypotheses \textnormal{\ref{hyp:H}} hold and carry out our analysis over the separable normed space $\Xpazo := C^0_c(\R^d,\C)$ of continuous and complex-valued observables with compact support.

\begin{definition}[Set-valued Koopman operators]
\label{def:Koopman}
We define the \textnormal{set-valued Koopman operators} associated with the control system \eqref{eq:ContCauchy} by
\begin{equation}
\label{eq:DefinitionKoopman}
\Kpazo_{(\tau,t)} : \varphi \in \Xpazo \tto \Big\{ \varphi \circ \Phi_{(\tau,t)}^u ~\,\textnormal{s.t. $u(\cdot) \in \Upazo$}  \Big\} \subset \Xpazo
\end{equation}
for all times $\tau,t \in[0,T]$.
\end{definition}

As highlighted in the introduction, the action of the set-valued Koopman operators can be thought of as outputting the reachable set of observables starting from $\varphi \in \Xpazo$ generated by all the admissible control signals. Before investigating these objects further, a few remarks are in order. First, notice that when the dynamics is uncontrolled, i.e. when there is a unique flow $(\Phi_t)_{t \in [0,T]} \subset C^0(\R^d,\R^d)$, then
\begin{equation*}
\Kpazo_{(\tau,t)}(\varphi) = \{\varphi \circ \Phi_{t-\tau}\} 
\end{equation*}
and the object introduced in Definition \ref{def:Koopman} coincides with the classical Koopman semigroup. Besides, it follows from  \eqref{eq:DefinitionKoopman} together with the fact that flows of nonautonomous dynamical systems form a semigroup for the composition operation that the latter satisfies the identity
\begin{equation*}
\Kpazo_{(\tau,t)}(\varphi) = \Kpazo_{(s,t)} \circ \Kpazo_{(\tau,s)}(\varphi)    
\end{equation*}
for all times $\tau,s,t \in [0,T]$, where we used the natural notation overloading
\begin{equation*}
\Kpazo_{(\tau,t)}(\BPhi) := \bigcup_{\varphi \in\BPhi} \Kpazo_{(\tau,t)}(\varphi)
\end{equation*}
for some set $\BPhi \subset\Xpazo$. Another direct observation that one can make is that the operator $\Kpazo_{(\tau,t)} : \Xpazo \tto \Xpazo$ is homogeneous, namely 
\begin{equation*}
\Kpazo_{(\tau,t)}(\alpha \varphi) = \alpha \Kpazo_{(\tau,t)}(\varphi)
\end{equation*}
for each $\alpha \in \R$. It is also subadditive, since one clearly has that
\begin{equation*}
\begin{aligned}
\Kpazo_{(\tau,t)} \big( \varphi_1 + \varphi_2 \big) & = \Big\{ \big( \varphi_1 + \varphi_2 \big) \circ \Phi_{(\tau,t)}^u ~\, \textnormal{s.t.}~ u(\cdot) \in \Upazo \Big\}  \\
& \subset \Big\{ \varphi_1 \circ \Phi_{(\tau,t)}^{u_1} + \varphi_2 \circ \Phi_{(\tau,t)}^{u_2} ~\, \textnormal{s.t.}~ u_1(\cdot),u_2(\cdot) \in \Upazo \Big\} \\
& = \Kpazo_{(\tau,t)}(\varphi_1) + \Kpazo_{(\tau,t)}(\varphi_2)
\end{aligned}
\end{equation*}
for each $\varphi_1,\varphi_2 \in \Xpazo$. However, the sets $\Kpazo_{(\tau,t)}(\varphi) \subset \Xpazo$ need not be convex in general, and the Koopman operators are thus not fans in sense of Definition \ref{def:ClosedProcess}. We shall see nonetheless that they still enjoy many nice properties, including an explicit characterisation of their adjoints along with a meaningful variant of the spectral mapping theorem, relating the eigenvalues of the semigroup and its infinitesimal generator.

\begin{remark}[Set-valued Koopman operators for differential inclusions]
We would like to stress that, while in the present work we define set-valued Koopman operators for controlled systems, most of our results remain valid for dynamical systems modelled more generally by differential inclusions, that is 
\begin{equation}
\label{eq:DiffInc}
\dot{x}(t) \in F(x(t))
\end{equation}
with $F : \R^d \tto \R^d$ being a Lipschitz continuous set-valued mapping with nonempty compact images. In that case, one can define the set-valued Koopman operators as 
\begin{equation*}
\big(\Kpazo_{(\tau,t)}(\varphi)\big)(x) := \bigg\{ \varphi(x(t)) ~\, \textnormal{s.t.}~ \text{$x(\cdot)$ solves \eqref{eq:DiffInc} with $x(\tau) = x \in \R^d$} \bigg\} 
\end{equation*} 
for all times $t \in [0,T]$. In the situation where $F : \R^d \tto \R^d$ has convex images -- an assumption that is quite standard in control theory, see e.g. \cite{Aubin1990,Clarke,Vinter}, and covered by Hypothesis \textnormal{\ref{hyp:C}} below --, the Lipschitz parametrisation theorem of \cite[Theorem 9.7.1]{Aubin1990} yields the existence of a compact set $U \subset \R^d$ along with a locally Lipschitz map $f:\R^d \times U \to \R^d$ such that $F(x) = \{ f(x,u) ~\, \textnormal{s.t.}~ u \in U \}$ for all $x \in \R^d$. 
\end{remark}

\begin{remark}[Concerning time-varying systems]
We also mention that the majority of our developments would still hold (up to some technical adaptations) for time-varying vector fields $(t,x,u) \in [0,T] \times \R^d \times U \mapsto f(t,x,u) \in \R^d$ which are $\Lcal^1$-measurable in $t \in [0,T]$. That being said, we chose to stick to the simpler framework of time-invariant control systems for the sake of readability, and point the interested reader to \cite{Frankowska1989} for some insight on how our results could be rephrased in this context. 
\end{remark}

In the following proposition, we establish basic regularity results for the set-valued Koopman operators, akin to those already known for their classical counterparts. 

\begin{proposition}[Regularity properties of set-valued Koopman operators]
Under Hypotheses \textnormal{\ref{hyp:H}}, the set-valued Koopman operators $\varphi \in \Xpazo \tto \Kpazo_{(\tau,t)}(\varphi)$ are $1$-Lipschitz for all times $\tau,t \in[0,T]$, and the set-valued map $(\tau,t) \in [0,T] \times [0,T] \tto \Kpazo_{(\tau,t)}(\varphi)$ is continuous for each $\varphi \in \Xpazo$. 
\end{proposition}

\begin{proof}
We start by studying the regularity of $\Kpazo_{(\tau,t)} : \Xpazo \tto \Xpazo$ for some given $\tau,t \in[0,T]$. We fix an arbitrary pair $\varphi_1,\varphi_2 \in \Xpazo$, and note that for every element $\psi_{(\tau,t)}^1 \in \Kpazo_{(\tau,t)}(\varphi_1)$, there exists some $u_1(\cdot)\in \Upazo$ such that $\psi_{(\tau,t)}^1 = \varphi_1 \circ \Phi_{(\tau,t)}^{u_1}$. Then, recalling that $(\Xpazo,\Norm{\cdot}_{\Xpazo})$ is a vector space, one has that
\begin{equation*}
\begin{aligned}
\psi_{(\tau,t)}^1 & = \varphi_1 \circ \Phi_{(\tau,t)}^{u_1} \\
& = \varphi_2 \circ \Phi_{(\tau,t)}^{u_1} + \big( \varphi_1 - \varphi_2 \big) \circ \Phi_{(\tau,t)}^{u_1} \\
& \in \Kpazo_{(\tau,t)}(\varphi_2) + \B_{\Xpazo} \big(0 ,  \Norm{\varphi_1 - \varphi_2}_{\Xpazo} \hspace{-0.1cm} \big),
\end{aligned}
\end{equation*}
which implies that $\Kpazo_{(\tau,t)}(\varphi_1) \subset \Kpazo_{(\tau,t)}(\varphi_2) + \B_{\Xpazo} \big( 0, \Norm{\varphi_1 - \varphi_2}_{\Xpazo})$ and is therefore tantamount to the $1$-Lipschitzianity of $\Kpazo_{(\tau,t)} : \Xpazo \tto \Xpazo$ for all times $\tau,t \in [0,T]$. Consider now an element $\varphi \in \Xpazo$, and observe that owing to the uniform bound displayed in \eqref{eq:StabEst} of Theorem \ref{thm:Well-posed}, the set  
\begin{equation*}
K_{\varphi} := \overline{\bigcup_{u(\cdot) \in \Upazo} \bigcup_{\tau,t \in[0,T]} \Phi_{(t,\tau)}^{u}(\supp(\varphi))} \subset \R^d
\end{equation*}
is well-defined and compact. Take any $\tau_1,t_1 \in[0,T]$, fix some $\psi_{(\tau_1,t_1)} \in \Kpazo_{(\tau_1,t_1)}(\varphi)$, and note that there exists a control $u_1(\cdot) \in \Upazo$ such that $\psi_{(\tau_1,t_1)} = \varphi \circ \Phi_{(\tau_1,t_1)}^{u_1}$. It then stems from the flow regularity estimate \eqref{eq:StabEst} that for any $\epsilon >0$, there exists some $\delta >0$ such that 
\begin{equation*}
\sup_{x \in K_{\varphi}} |\Phi_{(\tau_1,t_1)}^{u_1}(x) - \Phi^{u_1}_{(\tau,t)}(x)| \leq \epsilon
\end{equation*}
whenever $\tau,t \in [0,T]$ satisfy $(|\tau_1-\tau| + |t_1-t|) \leq \delta$. Observing in turn that $\varphi \in\Xpazo$ is uniformly continuous on its support while letting $\psi_{(\tau,t)} := \varphi \circ \Phi_{(\tau,t)}^{u_1} \in \Kpazo_{(\tau,t)}(\varphi)$, it then follows up to potentially choosing a smaller $\delta >0$ that
\begin{equation*}
\Norm{\psi_{(\tau_1,t_1)} - \psi_{(\tau,t)}}_{\Xpazo} \; = \sup_{x \in K_{\varphi}} \big| \varphi \circ \Phi_{(\tau_1,t_1)}^{u_1}(x) - \varphi \circ \Phi_{(\tau,t)}^{u_1}(x) \big| ~ \leq \, \epsilon.
\end{equation*}
Thence, the mapping $(\tau,t) \in [0,T] \times [0,T] \tto \Kpazo_{(\tau,t)}(\varphi) \subset \Xpazo$ is lower-semicontinuous at $(\tau_1,t_1) \in [0,T] \times [0,T]$ for every such pair, and thus over the whole time interval. It can be shown via similar arguments that it is also upper-semicontinuous.
\end{proof}

In what follows, we build on the previous result to prove a simple, yet enlightening representation formula for time-dependent Koopman observables.

\begin{proposition}[Representation formula for curves of Koopman observables] 
\label{prop:KoopmanRepresentation}
Given some $\varphi \in \Xpazo$, a function $(\tau,t) \in [0,T] \times [0,T] \to \psi_{(\tau,t)} \in \Xpazo$ is a measurable selection in the set-valued map $(\tau,t) \in [0,T] \times [0,T] \tto \Kpazo_{(\tau,t)}(\varphi)$ \textnormal{if and only if} there exists an $\Lcal^2$-measurable map $(\tau,t) \in [0,T] \times [0,T] \mapsto u_{\tau,t}(\cdot) \in \Upazo$ such that
\begin{equation*}
\psi_{(\tau,t)} = \varphi \circ \Phi_{(\tau,t)}^{u_{\tau,t}} 
\end{equation*}
for $\Lcal^2$-almost every $(\tau,t) \in [0,T] \times [0,T]$. 
\end{proposition}

\begin{proof}[Proof of Proposition \ref{prop:KoopmanRepresentation}]
Clearly, if $\psi_{(\tau,t)} = \varphi \circ \Phi_{(\tau,t)}^{u_{\tau,t}}$ for some measurable map $(\tau,t) \in [0,T] \times [0,T] \mapsto u_{\tau,t}(\cdot) \in \Upazo$, then $\psi_{(\tau,t)} \in \Kpazo_{(\tau,t)}(\varphi)$ for $\Lcal^2$-almost every $(\tau,t) \in [0,T] \times [0,T]$. Moreover, the map $(\tau,t) \in [0,T] \times [0,T] \mapsto \psi_{(\tau,t)} \in \Xpazo$ is then measurable by the continuity of 
\begin{equation*}
(\tau,t,u(\cdot)) \in [0,T] \times [0,T] \times \Upazo \mapsto \Phi_{(\tau,t)}^u \in C^0(\R^d,\R^d)   
\end{equation*}
that stems from Theorem \ref{thm:Well-posed} and Corollary \ref{cor:ContinuousFlow}. Suppose now that $(\tau,t) \in [0,T] \times [0,T] \mapsto \psi_{(\tau,t)} \in \Kpazo_{(\tau,t)}(\varphi)$ is an arbitrary measurable selection in the set-valued Koopman operator, and notice that the latter can be expressed as 
\begin{equation*}
\Kpazo_{(\tau,t)}(\varphi) = \Big\{ \Psi_{\varphi}(\tau,t,u(\cdot)) ~\,\textnormal{s.t.}~ u(\cdot) \in \Upazo \Big\}
\end{equation*}
where $\Psi_{\varphi} : (\tau,t,u(\cdot)) \in [0,T] \times [0,T] \times \Upazo \mapsto \varphi \circ \Phi_{(\tau,t)}^u \in \Xpazo$ is continuous, again as a consequence of Theorem \ref{thm:Well-posed} and Corollary \ref{cor:ContinuousFlow}. Because the completion of $(\Xpazo,\Norm{\cdot}_{\Xpazo})$ is a separable Banach space, we may apply the selection principle of Theorem \ref{thm:FilippovSel} to infer the existence of a measurable map $(\tau,t) \in [0,T] \times [0,T] \mapsto u_{\tau,t}(\cdot) \in \Upazo$ such that 
\begin{equation*}
\psi_{(\tau,t)} = \Psi_{\varphi}(\tau,t,u_{\tau,t}(\cdot)) = \varphi \circ \Phi_{(\tau,t)}^{u_{\tau,t}}
\end{equation*}
for $\Lcal^2$-almost every $\tau,t \in [0,T]$, which closes the proof.
\end{proof}

\begin{remark}[On the representation of Koopman observables]
In essence, the previous result shows that while every element in $\Kpazo_{(\tau,t)}(\varphi)$ can be expressed using a single controlled flow as long as both $\tau,t \in [0,T]$ are fixed, a time-dependent family $(\tau,t) \in [0,T] \times [0,T] \mapsto \psi_{(\tau,t)} \in \Kpazo_{(\tau,t)}(\varphi)$ may a priori jump between the realisations of different controlled dynamics, as illustrated in Figure \ref{fig:MeasurableSel} below.
\end{remark}

\begin{figure}
\centering
\resizebox{0.9\textwidth}{!}{
\begin{tikzpicture}
\draw[->] (-0.25,0)--(2.75,0);
\draw[->] (0,-0.25)--(0,1.75);
\draw[->] (0.15,0.15)--(-1,-1);
\draw (-0.5,-0.85) node {\scriptsize $\R^d$};
\draw (-0.25,1.5) node {\scriptsize $\C$};
\filldraw[opacity = 0.8, draw = black, fill = white, thin, bottom color = white, top color = gray!40] plot [smooth, tension=0.6] coordinates {(2.5,-0.4)(2.25,-0.35)(1.45,-0.5)(0.75,-0.35)(0.25,-0.435)(0.75,-1)(1.45,-1.25)(2.25,-1.1)};
\draw[black, dashed, line width =0.05] plot [smooth, tension=0.8] coordinates {(0.25,-0.435)(0.75,-0.45)(1.45,-0.65)(2.25,-0.5)(2.45,-0.55)};
\draw[black, dashed, line width =0.05] plot [smooth, tension=0.8] coordinates {(0.25,-0.435)(0.75,-0.6)(1.45,-0.85)(2.25,-0.7)(2.35,-0.7)};
\draw[black, line width =0.5] plot [smooth, tension=0.8] coordinates {(0.25,-0.435)(0.75,-0.8)(1.45,-1.05)(2.3,-0.9)};
\filldraw[opacity = 0.5, draw = gray, fill = white, thin, bottom color = white, top color = gray!40] plot [smooth, tension=0.8] coordinates {(2.5,0.7)(2.25,0.8)(1.455,0.7)(0.75,0.85)(0.25,1)(0.98,0.35)(1.8,0.15)(2.25,0.2)};
\draw[black, line width =0.5] plot [smooth, tension=0.8] coordinates {(0.275,1)(0.75,0.6)(1.45,0.35)(2.35,0.35)};
\draw[black, dashed] (0.25,-0.45)--(0.25,1);
\draw[black, dashed] (1,-0.95)--(1,0.455);
\draw[black, dashed] (2,-1)--(2,0.315);
\draw[black] (0.25,1) node {\LARGE $\cdot$};
\draw[black] (1,0.455) node {\LARGE $\cdot$};
\draw[black] (2,0.315) node {\LARGE $\cdot$};
\draw[black] (0.6,1.25) node {\scriptsize $\varphi(x)$};
\draw[black] (1.5,0.95) node {\scriptsize $\psi_{(0,t_1)}(x)$};
\draw[black] (2.5,0.545) node {\scriptsize $\psi_{(0,t_2)}(x)$};
\draw (0.25,-0.455) node {\LARGE $\cdot$};
\draw (1,-0.95) node {\LARGE $\cdot$};
\draw (2,-1) node {\LARGE $\cdot$};
\draw (0.1,-0.6) node {\scriptsize $x$};
\draw (0.75,-1.5) node {\scriptsize $\Phi_{(0,t_1)}^u(x)$};
\draw (2.5,-1.5) node {\scriptsize $\Phi_{(0,t_2)}^u(x)$};
\begin{scope}[xshift=6cm]
\draw[->] (-0.25,0)--(2.75,0);
\draw[->] (0,-0.25)--(0,1.75);
\draw[->] (0.15,0.15)--(-1,-1);
\draw (-0.5,-0.85) node {\scriptsize $\R^d$};
\draw (-0.25,1.5) node {\scriptsize $\C$};
\filldraw[opacity = 0.5, draw = gray, fill = white, thin, bottom color = white, top color = gray!40] plot [smooth, tension=0.8] coordinates {(2.5,0.7)(2.25,0.8)(1.455,0.7)(0.75,0.85)(0.25,1)(0.98,0.35)(1.8,0.15)(2.25,0.2)};
\filldraw[opacity = 0.8, draw = black, fill = white, thin, bottom color = white, top color = gray!40] plot [smooth, tension=0.6] coordinates {(2.5,-0.4)(2.25,-0.35)(1.45,-0.5)(0.75,-0.35)(0.25,-0.435)(0.75,-1)(1.45,-1.25)(2.25,-1.1)};
\draw[black, dashed, line width =0.05] plot [smooth, tension=0.8] coordinates {(0.25,-0.435)(0.75,-0.45)(1.45,-0.65)(2.25,-0.5)(2.45,-0.55)};
\draw[black, dashed, line width =0.05] plot [smooth, tension=0.8] coordinates {(0.25,-0.435)(0.75,-0.6)(1.45,-0.85)(2.25,-0.7)(2.35,-0.725)};
\draw[black, dashed, line width =0.05] plot [smooth, tension=0.8] coordinates {(0.25,-0.435)(0.75,-0.8)(1.45,-1.05)(2.3,-0.9)};
\draw[black, line width =0.5] plot [smooth, tension=0.8] coordinates {(0.25,-0.43)(0.75,-0.45)(1.25,-0.635)};
\draw[black, line width =0.5] plot [smooth, tension=0.8] coordinates {(1.05,-0.735)(1.45,-0.855)(1.8,-0.82)};
\draw[black, line width =0.5] plot [smooth, tension=0.8] coordinates {(1.7,-1.04)(2,-0.98)(2.3,-0.9)};
\draw[black, line width =0.5] plot [smooth, tension=0.8] coordinates {(0.25,1.02)(0.75,0.75)(1.25,0.6)};
\draw[black, line width =0.5] plot [smooth, tension=0.8] coordinates {(1.05,0.5)(1.45,0.425)(1.8,0.5)};
\draw[black, line width =0.5] plot [smooth, tension=0.8] coordinates {(1.7,0.35)(2,0.425)(2.3,0.4)};
\draw[black, dashed] (0.25,-0.45)--(0.25,1);
\draw[black, dashed] (1.05,-0.745)--(1.05,0.485);
\draw[black, dashed] (2,-1)--(2,0.41);
\draw[black] (0.25,1) node {\LARGE $\cdot$};
\draw[black] (1.05,0.485) node {\LARGE $\cdot$};
\draw[black] (2,0.41) node {\LARGE $\cdot$};
\draw[black] (0.6,1.25) node {\scriptsize $\varphi(x)$};
\draw[black] (1.5,0.95) node {\scriptsize $\psi_{(0,t_1)}(x)$};
\draw[black] (2.6,0.6) node {\scriptsize $\psi_{(0,t_2)}(x)$};
\draw (0.25,-0.455) node {\LARGE $\cdot$};
\draw (1.05,-0.745) node {\LARGE $\cdot$};
\draw (2,-1) node {\LARGE $\cdot$};
\draw (0.1,-0.6) node {\scriptsize $x$};
\draw (0.75,-1.5) node {\scriptsize $\Phi_{(0,t_1)}^{u_1}(x)$ \hspace{0.2cm}$\dots$};
\draw (2.55,-1.5) node {\scriptsize $\Phi_{(0,t_2)}^{u_3}(x)$};
\end{scope}
\end{tikzpicture}
}
\caption{{\small \textit{Representation of a particular curve of Koopman observables corresponding to measurements evaluated along a single controlled flow (left) and of a general curve of observables that may a priori jump in a measurable way between different controlled flows (right).}}}
\label{fig:MeasurableSel}
\end{figure}

Throughout the manuscript, many of our results will rely on the assumption that the set of controlled velocities is convex, in a way which captures the fact that Koopman operators act globally -- and not pointwisely -- on the underlying state space. 

\begin{taggedhypsing}{\textbn{(C)}}
\label{hyp:C}
The set of controlled vector fields $\Fpazo:= \big\{ f_u \in C^0(\R^d,\R^d) ~\,\textnormal{s.t.}~ u \in U \big\}$ is convex.  
\end{taggedhypsing}

A prototypical example of admissible velocity set satisfying Hypotheses \textnormal{\ref{hyp:H}} and \textnormal{\ref{hyp:C}} above is given by nonlinear control-affine systems of the form 
\begin{equation*}
f_u(x) := f_0(x) + \sum_{k=1}^n u_k f_k(x),
\end{equation*}
where $U \subset \R^n$ is convex and $f_0,\dots,f_n \in C^0(\R^d,\R^d)$ are locally Lipschitz vector fields with sublinear growth. It should be stressed however that Hypothesis \textnormal{\ref{hyp:C}} is in general much stronger than the usual condition that $\{f(x,u) ~\, \textnormal{s.t.}~ u \in U \} \subset \R^d$ be a convex set for all fixed $x \in \R^d$.

\medskip

In the ensuing proposition, we show that Koopman operators have compact images whenever the set of admissible controlled vector fields is convex in the above sense. 

\begin{proposition}[Topological properties of Koopman operators]
\label{prop:KoopmanTopo}
Suppose that Hypotheses \textnormal{\ref{hyp:H}} and \textnormal{\ref{hyp:C}} hold. Then, the sets $\Kpazo_{(\tau,t)}(\varphi) \subset\Xpazo$ are compact for all times $\tau,t \in [0,T]$ and each $\varphi \in \Xpazo$, and the operators $\Kpazo_{(\tau,t)} : \Xpazo \tto \Xpazo$ are closed processes. 
\end{proposition}

\begin{proof}
To show that, given $\tau,t \in [0,T]$ and some $\varphi \in\Xpazo$, the sets $\Kpazo_{(\tau,t)}(\varphi)$ are compact, fix $(\psi_{(\tau,t)}^n) \subset \Kpazo_{(\tau,t)}(\varphi)$ and let $(u_n(\cdot)) \subset \Upazo$ be the sequence of signals such that $\psi_{(\tau,t)}^n = \varphi \circ\Phi_{(\tau,t)}^{u_n}$ for each $n \geq 1$. Recall also that by Theorem \ref{thm:Well-posed}, there exist for each $R>0$ a pair of constants $M_R,L_R>0$ such that
\begin{equation}
\label{eq:FlowEstKoopman}
\big| \Phi_{(\tau,t)}^{u_n}(x) \big| \leq M_R \qquad \text{and} \qquad \big| \Phi_{(\tau,t)}^{u_n}(x) - \Phi_{(\tau,t)}^{u_n}(y) \big| \leq L_R |x-y|,
\end{equation}
for $x,y \in B(0,R)$ and each $n \geq 1$. Therefore, the sequence of maps $(\Phi_{(\tau,t)}^{u_n}) \subset C^0(\R^d,\R^d)$ is locally uniformly valued in a compact set and locally equi-Lipschitz, which by the Ascoli-Arzel\`a theorem (see e.g. \cite[Chapter 7 Theorem 18]{Kelley1975}) yields the existence of an element $\Phi_{(\tau,t)} \in C^0(\R^d,\R^d)$ such that 
\begin{equation}
\label{eq:FlowConv}
\big\| \Phi_{(\tau,t)} - \Phi_{(\tau,t)}^{u_{n_k}} \big\|_{C^0(K,\R^d)} ~\underset{k \to+\infty}{\longrightarrow}~ 0
\end{equation}
for each compact set $K \subset \R^d$, along a subsequence $(u_{n_k}(\cdot)) \subset \Upazo$. Observing now that 
\begin{equation*}
K_{\varphi} := \overline{\bigcup_{n \geq 1} \bigcup_{\tau,t \in [0,T]} \Phi^{u_n}_{(t,\tau)}(\supp(\varphi))}
\end{equation*}
is a compact set as a consequence of \eqref{eq:FlowEstKoopman}, it necessarily follows that
\begin{equation*}
\Norm{\varphi\circ\Phi_{(\tau,t)} - \psi_{(\tau,t)}^{n_k}}_{\Xpazo} ~\underset{k \to +\infty}{\longrightarrow}~ 0
\end{equation*}
along the same subsequence. Hence, to conclude, there remains to show that
\begin{equation}
\label{eq:FlowControlEq}
\Phi_{(\tau,t)} = \Phi_{(\tau,t)}^u    
\end{equation}
for some admissible control $u(\cdot) \in \Upazo$. To this end, recall first that by Proposition \ref{prop:Compact}, the set of admissible velocities $\Fpazo \subset C^0(\R^d,\R^d)$ is convex and compact for the topology of local uniform convergence. Thus, for each compact set $K \subset \R^d$, the collection of restricted vector fields
\begin{equation*}
\Fpazo_K := \Big\{ f_{|K} \in C^0(K,\R^d) ~\, \textnormal{s.t.}~ f \in \Fpazo \Big\}
\end{equation*}
is a compact and convex set as well. Therefore, we can inductively apply the compactness criterion of Proposition \ref{prop:WeakCompactness} on a countably increasing family of balls covering $\R^d$ along with a diagonal argument to obtain the existence of an $\Lcal^1$-measurable map $s \in [0,T] \mapsto f(s) \in \Fpazo$ satisfying
\begin{equation}
\label{eq:WeakConvergenceVelocity}
\INTSeg{\INTDom{\zeta \cdot \Big( f(s,x) - f_{u_{n_k}(s)}(x) \Big)}{\R^d}{\mu(s)(x)}}{s}{0}{T} ~\underset{k \to +\infty}{\longrightarrow}~ 0 
\end{equation}
for every $\mu : [0,T] \to \Mcal(K,\C)$ complying with the conditions of Proposition \ref{prop:WeakCompactness} for $\Xpazo := C^0(K,\R^d)$, each $\zeta \in \R^d$ and every compact set $K \subset \R^d$. Therein, the convergence holds along a subsequence that may depend on the set $K \subset \R^d$ itself, and that we do not relabel. At this stage, we consider the set
\begin{equation*}
K_x := \overline{\bigcup_{n \geq 1} \bigcup_{\tau,s \in [0,T]} \Phi^{u_n}_{(\tau,s)}(x)}, 
\end{equation*}
defined for any $x \in \R^d$, and note that it is compact, again by \eqref{eq:FlowEstKoopman}. Setting now
\begin{equation*}
\mu(s) := \mathds{1}_{[\tau,t]}(s) \, \delta_{\Phi_{(\tau,s)}(x)} \in \Mcal(K_x,\R)
\end{equation*}
for $\Lcal^1$-almost every $s \in [0,T]$, it follows from \eqref{eq:WeakConvergenceVelocity} that 
\begin{equation*}
\INTSeg{\zeta \cdot \Big( f(s,\Phi_{(\tau,s)}(x)) - f_{u_{n_k}(s)} \big(\Phi_{(\tau,s)}(x) \big) \Big)}{s}{\tau}{t} ~\underset{k \to +\infty}{\longrightarrow}~ 0
\end{equation*}
for every $\zeta\in\R^d$, along a subsequence that may depend on $(t,x) \in [0,T] \times \R^d$. Besides
\begin{equation*}
\INTSeg{\,\NormC{ f_{u_{n_k}(s)} \big( \Phi_{(\tau,s)} \big) - f_{u_{n_k}(s)} \big( \Phi_{(\tau,s)}^{u_{n_k}} \big)}{0}{K_x,\R^d}}{s}{\tau}{t} ~\underset{k \to +\infty}{\longrightarrow}~ 0
\end{equation*}
as a byproduct of \eqref{eq:FlowConv} and Hypothesis \textnormal{\ref{hyp:H}}-$(ii)$, from whence we can conclude by merging the previous two equations that 
\begin{equation}
\label{eq:WeakConvergenceVelocityBis}
\INTSeg{f_{u_{n_k}(s) } \Big(\Phi_{(\tau,s)}^{u_{n_k}}(x) \Big)}{s}{\tau}{t} ~\underset{k \to +\infty}{\longrightarrow}~ \INTSeg{f \Big(s,\Phi_{(\tau,s)}(x) \Big)}{s}{\tau}{t}
\end{equation}
along a subsequence that may depend on $(t,x) \in [0,T] \times \R^d$. Since the latter are arbitrary, we can deduce from \eqref{eq:FlowConv} and \eqref{eq:WeakConvergenceVelocityBis} that
\begin{equation}
\label{eq:EquationPhit}
\Phi_{(\tau,t)}(x) = x + \INTSeg{f \Big(s,\Phi_{(\tau,s)}(x)\Big)}{s}{\tau}{t}
\end{equation}
for all $(t,x) \in [0,T] \times \R^d$. Finally, upon noting that  
\begin{equation*}
f(s) \in \Fpazo = \Big\{ f_u \in C^0(\R^d,\R^d) ~\, \textnormal{s.t.}~ u \in U \Big\}
\end{equation*}
for $\Lcal^1$-almost every $s \in[0,T]$, while observing that the map $u \in U \mapsto f_u \in C^0(\R^d,\R^d)$ is continuous under Hypotheses \textnormal{\ref{hyp:H}}, we may apply Theorem \ref{thm:FilippovSel} to infer the existence of some control signal $u(\cdot) \in \Upazo$ such that
\begin{equation*}
f(s) = f_{u(s)}
\end{equation*}
for $\Lcal^1$-almost every $s \in [0,T]$. This together with \eqref{eq:EquationPhit} amounts to \eqref{eq:FlowControlEq}, and thus concludes the proof of our first claim.

We now prove that $\Kpazo_{(\tau,t)} : \Xpazo \tto \Xpazo$ is a closed process for all times $\tau,t \in[0,T]$, namely that its graph is a closed cone. The latter of these properties is straightforward, since for each $\alpha>0$ and every $(\varphi,\psi_{(\tau,t)}) \in\Graph(\Kpazo_{(\tau,t)})$, one has that 
\begin{equation*}
\alpha \psi_{(\tau,t)} = \alpha \varphi \circ\Phi_{(\tau,t)}^u \in \Kpazo_{(\tau,t)}(\alpha \varphi)
\end{equation*}
where $u(\cdot) \in \Upazo$. In order to prove that $\Graph(\Kpazo_{(\tau,t)})$ is closed, we consider a sequence $((\varphi_n,\psi_{(\tau,t)}^n)) \subset \Graph(\Kpazo_{(\tau,t)})$ such that 
\begin{equation}
\label{eq:ClosedGraphKoopman}
\Norm{\varphi - \varphi_n}_{\Xpazo} \underset{n\to + \infty}{\longrightarrow} 0 \qquad \text{and} \qquad \Norm{\psi - \psi_{(\tau,t)}^n}_{\Xpazo} \underset{n\to + \infty}{\longrightarrow} 0,       
\end{equation}
and recall that by definition, there exists $(u_n(\cdot)) \subset\Upazo$ such that
\begin{equation*}
\psi_{(\tau,t)}^n = \varphi_n \circ \Phi_{(\tau,t)}^{u_n} 
\end{equation*}
for each $n\geq1$. Then, upon repeating the compactness argument detailed above combined with \eqref{eq:ClosedGraphKoopman}, it follows that $\psi \in \Kpazo_{(\tau,t)}(\varphi)$, which concludes the proof.  
\end{proof}


\section{Set-valued Liouville and Perron-Frobenius operators}
\label{section:LiouvillePerron}

In this section, we define the set-valued counterparts of the Liouville and Perron-Frobenius operators, for which we provide several structure results and explicit characterisations. In what follows, we denote by $\Dpazo:= C^1_c(\R^d,\C)$ the separable normed space of continuously differentiable complex-valued functions with compact support.


\subsection{Set-valued Liouville operators and Koopman dynamics}
 \label{subsection:Liouville}

In what ensues, we define the set-valued Liouville operator as the collection of all Liouville operators associated with each individual controls, and show that the latter can be rigorously related to the infinitesimal behaviour of the semigroup at time $t = \tau$. 

\begin{definition}[Set-valued Liouville operators]\label{def:L}
We define the \textnormal{set-valued Liouville operator} $\Lpazo : \Dpazo \tto \Xpazo$ associated with \eqref{eq:ContCauchy} as
\begin{equation}
\label{eq:LiouvilleDef}
\Lpazo(\varphi) := \Big\{ \nabla_x \varphi \cdot f_u ~\, \textnormal{s.t.}~ u \in U \Big\} \subset \Xpazo
\end{equation}
for each $\varphi \in \Dpazo$. 
\end{definition}

We start by establishing some of the simple topological properties of these Liouville operators. These results will prove insightful when studying its point spectrum further down in Section \ref{subsection:Spectral}.

\begin{proposition}[Topological properties of the Liouville operator]
Suppose that Hypotheses \textnormal{\ref{hyp:H}} and \textnormal{\ref{hyp:C}} hold. Then, the set-valued Liouville operator $\Lpazo : \Dpazo \tto \Xpazo$ is a fan with compact images.
\end{proposition}

\begin{proof}
The fact that $\Lpazo : \Dpazo \tto \Xpazo$ has convex and compact images is a direct consequence of Hypothesis \textnormal{\ref{hyp:C}}, on the one hand, and of Hypotheses \textnormal{\ref{hyp:H}} together with Proposition \ref{prop:Compact} on the other hand. Besides, it follows from its very definition that $0 \in \Lpazo(0)$ and
\begin{equation*}
\begin{aligned}
\Lpazo(\varphi_1 + \varphi_2) & = \Big\{ \big( \nabla_x \varphi_1 + \nabla_x \varphi_2 \big) \cdot f_u ~\, \textnormal{s.t.}~ u \in U \Big\} \\
& \subset \Big\{ \nabla_x \varphi_1 \cdot f_{u_1} + \nabla_x \varphi_2 \cdot f_{u_2} ~\, \textnormal{s.t.}~ u_1,u_2 \in U \Big\} =  \Lpazo(\varphi_1) + \Lpazo(\varphi_2) 
\end{aligned}
\end{equation*}
for each $\varphi_1,\varphi_2 \in \Dpazo$. Hence, there simply remains to show that $\Lpazo : \Dpazo \tto \Xpazo$ is a closed process, namely that its graph is a closed cone. The latter of these properties is straightforward, as for each $\alpha >0$ and every $(\varphi,\psi) \in \Graph(\Lpazo)$, it holds that
\begin{equation*}
\alpha \psi = \alpha \nabla_x\varphi \cdot f_u \in \Lpazo(\alpha\varphi)
\end{equation*}
where $u \in U$ is some fixed control value. Suppose now that we are given a sequence $((\varphi_n,\psi_n)) \subset \Graph(\Lpazo)$ such that 
\begin{equation}
\label{eq:ConvergenceClosedGraph}
\Norm{\varphi - \varphi_n}_{\Dpazo} ~\underset{n \to +\infty}{\longrightarrow}~ 0 \qquad \text{and} \qquad \Norm{\psi - \psi_n}_{\Xpazo} ~\underset{n \to +\infty}{\longrightarrow}~ 0.
\end{equation}
This fact together with the definition \eqref{eq:LiouvilleDef} of the Liouville operator yield the existence of a sequence $(u_n) \subset U$ for which
\begin{equation*}
\psi_n = \nabla_x\varphi_n \cdot f_{u_n} 
\end{equation*}
for all $n \geq 1$. The conclusion simply follows then from the compactness of the set of admissible velocities established in Proposition \ref{prop:Compact}. 
\end{proof}

In the following theorem, we prove a result which establishes an explicit connection between the Liouville operator and adequate set-valued derivatives of the Koopman semigroup, in the spirit of \cite[Section 2]{Frankowska1995}.

\begin{theorem}[Generator of the Koopman semigroups]
\label{thm:RepresentationLiouville}
Suppose that Hypotheses \textnormal{\ref{hyp:H}} hold and fix a time $\tau \in [0,T]$. Then, the following inclusions 
\begin{equation}
\label{eq:LiminfInclusionLiouville}
\Lpazo(\varphi) \subset \Liminf{t \to \tau} \, \frac{\Kpazo_{(\tau,t)}(\varphi)-\varphi}{t-\tau}
\end{equation}
and 
\begin{equation}
\label{eq:LimsupInclusionLiouville}    
\Limsup{t \to \tau} \, \frac{\Kpazo_{(\tau,t)}(\varphi)-\varphi}{t-\tau} \subset \co \Lpazo(\varphi)
\end{equation}
hold for each $\varphi \in \Dpazo$. Furthermore, if Hypothesis \textnormal{\ref{hyp:C}} holds, then the set-valued Liouville operator $\Lpazo : \Dpazo \tto \Xpazo$ is the infinitesimal generator of the Koopman semigroup, in the sense that
\begin{equation}
\label{eq:LiouvilleCharac}
\Lpazo(\varphi) = \Lim{t \to \tau} \frac{\Kpazo_{(\tau,t)}(\varphi)-\varphi}{t-\tau}
\end{equation}
for each $\varphi \in \Dpazo$. 
\end{theorem}

\begin{proof}
It is quite clear from the definitions of lower and upper limits of sequences of sets provided in Section \ref{subsection:Setvalued} above that
\begin{equation*}
\Liminf{t \to \tau} \, \frac{\Kpazo_{(\tau,t)}(\varphi)-\varphi}{t-\tau} \subset \Limsup{t \to \tau} \, \frac{\Kpazo_{(\tau,t)}(\varphi)-\varphi}{t-\tau}
\end{equation*}
for each $\varphi \in \Dpazo$. Remarking thus that, under Hypothesis \textnormal{\ref{hyp:C}}, there holds
\begin{equation*}
\co \Lpazo(\varphi) = \co \Big\{ \nabla_x \varphi \cdot f ~\, \textnormal{s.t.}~ f \in \Fpazo \Big\} = \Big\{ \nabla_x \varphi \cdot f ~\, \textnormal{s.t.}~ f \in \co \Fpazo \Big\} = \Lpazo(\varphi),
\end{equation*}
it would directly follow from \eqref{eq:LiminfInclusionLiouville} and \eqref{eq:LimsupInclusionLiouville} that the Liouville operator coincides with the set-valued derivative displayed in \eqref{eq:LiouvilleCharac}. 

We thus start by establishing the liminf inclusion \eqref{eq:LiminfInclusionLiouville}, which amounts to finding for each element $\bar{u} \in U$ a curve of  observables $t \in [0,T] \mapsto \psi_{(\tau,t)} \in \Kpazo_{(\tau,t)}(\varphi)$ such that
\begin{equation}
\label{eq:LiminfInc}
\lim_{t \to \tau} \frac{\psi_{(\tau,t)} - \varphi}{t-\tau} = \nabla_x \varphi \cdot f_{\bar{u}}. 
\end{equation}
To this end, consider the constant signal given by $u(t) := \bar{u}$ for all times $t \in [0,T]$, and let $\psi_{(\tau,t)} := \varphi \in \Phi_{(\tau,t)}^u$. Observe then that 
\begin{equation*}
\begin{aligned}
\Phi_{(\tau,t)}^u(x) & = x + \INTSeg{f \Big( \Phi_{(\tau,s)}^{u}(x) , \bar{u} \Big)}{s}{\tau}{t} \\
& = x + (t-\tau) f_{\bar{u}}(x) + \INTSeg{\bigg( f_{\bar{u}} \big( \Phi_{(\tau,s)}^{u}(x) \big) - f_{\bar{u}}(x) \bigg)}{s}{\tau}{t}, 
\end{aligned}
\end{equation*}
for all $x \in\R^d$. By combining Hypothesis \textnormal{\ref{hyp:H}}-$(ii)$ and the Lipschitz estimate of \eqref{eq:StabEst} in Theorem \ref{thm:Well-posed} within the latter identity, one shows that
\begin{equation}
\label{eq:LiminfTaylor}
\sup_{x \in K_{\varphi}} \big| \Phi_{(\tau,t)}^u(x) - x - (t-\tau) f_{\bar{u}}(x) \big| \leq C_{\varphi} |t-\tau|^2
\end{equation}
for all times $t \in [0,T]$, where $C_{\varphi} > 0$ is a constant and $K_{\varphi} \subset \R^d$ is the compact set 
\begin{equation}
\label{eq:KphiDef}
K_{\varphi} := \overline{\bigcup_{u(\cdot) \in \Upazo} \bigcup_{\tau,t \in[0,T]} \Phi_{(t,\tau)}^u \big(\supp(\varphi) \big)},
\end{equation}
whose existence is guaranteed by \eqref{eq:StabEst}. Thence, recalling that $\varphi \in C^1_c(\R^d,\C)$, it follows from \eqref{eq:LiminfTaylor} combined with \eqref{eq:KphiDef} that
\begin{equation*}
\lim_{t \to \tau} \Bigg\| \frac{\varphi \circ \Phi_{(\tau,t)}^u - \varphi - (t-\tau) \nabla_x \varphi \cdot f_{\bar{u}}}{t-\tau} \Bigg\|_{\Xpazo}  = 0,  
\end{equation*}
which then yields \eqref{eq:LiminfInc} by what precedes, since $\bar{u} \in U$ is arbitrary. 

To complete the proof, there remains to derive the limsup inclusion \eqref{eq:LimsupInclusionLiouville}, which can be equivalently recast as the requirement that for every pair of sequences $(t_n) \subset [0,T] \setminus \{\tau\}$ and $(\psi_{(\tau,t_n)}) \subset \Xpazo$ satisfying $\psi_{(\tau,t_n)} \in \Kpazo_{(\tau,t_n)}(\varphi)$ for each $n \geq 1$, there exists a subsequence $t_{n_k} \to\, \tau$ such that 
\begin{equation}
\label{eq:LimsupInc}
\lim_{t_{n_k} \to \, \tau} \hspace{-0.05cm} \frac{\psi_{(\tau,t_{n_k})} - \varphi}{t_{n_k}-\tau} \in \co \Lpazo(\varphi). 
\end{equation}
Fixing such a pair, it follows from Proposition \ref{prop:KoopmanRepresentation} that  
\begin{equation*}
\psi_{(\tau,t_n)} = \varphi \circ \Phi_{(\tau,t_n)}^{u_n}
\end{equation*}
for each $n \geq 1$ and some $(u_n(\cdot)) \subset \Upazo$. Then, by definition of the characteristic flows 
\begin{equation}
\label{eq:Limsup1}
\begin{aligned}
\Phi_{(\tau,t_n)}^{u_n}(x) & = x + \INTSeg{f \Big( \Phi_{(\tau,s)}^{u_n}(x) , u_n(s) \Big)}{s}{\tau}{t_n} \\
& = x + (t_n-\tau) \bigg( \frac{1}{t_n-\tau} \INTSeg{f(x, u_n(s))}{s}{\tau}{t_n} \bigg) \\
& \hspace{0.7cm} + \INTSeg{\bigg( f \Big( \Phi_{(\tau,s)}^{u_n}(x) , u_n(s) \Big) - f(x,u_n(s)) \bigg)}{s}{\tau}{t_n}
\end{aligned}
\end{equation}
for each $n \geq 1$. Denoting by $K_{\varphi} \subset\R^d$ the compact set associated with $\varphi \in \Xpazo$ via \eqref{eq:KphiDef}, it holds that
\begin{equation}
\label{eq:Limsup2}
\sup_{x \in K_{\varphi}} \bigg| \INTSeg{\bigg( f \Big( \Phi_{(\tau,s)}^{u_n}(x) , u_n(s) \Big) - f(x,u_n(s)) \bigg)}{s}{\tau}{t_n} \bigg| \leq C_{\varphi} |t_n-\tau|^2, 
\end{equation}
for each $n \geq 1$ and some constant $C_{\varphi} >0$. Besides, regarding for any given $n \geq 1$ and each compact set $K \subset \R^d$ the map $t \in[0,T] \mapsto f_{u_n(t)} \in C^0(K,\R^d)$ as an element of the Bochner space $L^1([0,T],C^0(K,\R^d))$, it follows from the range convexity of the Bochner integral (see e.g. \cite[Proposition 1.2.12]{AnalysisBanachSpaces}) that
\begin{equation*}
\frac{1}{t_n-\tau} \INTSeg{f_{u_n(s)}}{s}{\tau}{t_n} \in \co \Fpazo
\end{equation*}
for each $n \geq 1$. Under Hypotheses \textnormal{\ref{hyp:H}}, it can further be checked that the mappings 
\begin{equation*}
F_n : x \in K_{\varphi} \mapsto \frac{1}{t_n-\tau} \INTSeg{f_{u_n(s)}(x)}{s}{\tau}{t_n} \in \R^d
\end{equation*}
form a uniformly bounded and equi-Lipschitz family in $C^0(K_{\varphi},\R^d)$, which, by the standard Ascoli-Arzel\`a theorem (see e.g. \cite[Chapter 7 Theorem 18]{Kelley1975}) combined with the fact that $\co \Fpazo \subset C^0(\R^d,\R^d)$ is compact as a consequence of Proposition \ref{prop:Compact} together with \cite[Theorem 5.35]{Aliprantis2006}, entails the existence of some $F \in \co \Fpazo$ for which
\begin{equation}
\label{eq:Limsup3}
\lim_{k \to+\infty} \NormC{F - F_{n_k}}{0}{K_{\varphi},\R^d} = 0   
\end{equation}
along a subsequence $n_k \to +\infty$. Owing to the fact that $\varphi\in C^1_c(\R^d,\C)$, it then follows from \eqref{eq:Limsup1}, \eqref{eq:Limsup2} and \eqref{eq:Limsup3} that 
\begin{equation*}
\psi_{(\tau,t_{n_k})} = \varphi \circ \Phi^{u_{n_k}}_{(\tau,t_{n_k})} = \varphi + (t_{n_k}-\tau) \nabla_x \varphi \cdot F + o(|t_{n_k}-\tau|)
\end{equation*}
when $t_{n_k} \to \,\tau$. This last identity then yields
\begin{equation*}
\lim_{t_{n_k} \to \,\tau} \frac{\psi_{(\tau,t_{n_k})} - \varphi}{t_{n_k}-\tau} = \nabla_x \varphi \cdot F \in \co \Lpazo(\varphi)
\end{equation*}
which closes the proof of Theorem \ref{thm:RepresentationLiouville}.
\end{proof}

\begin{remark}[On the relation between the set-valued Koopman and Liouville operators]
The previous theorem provides us with two insights on the structure of the Koopman and Liouville operators for control systems. Firstly, the inclusions \eqref{eq:LiminfInclusionLiouville} and \eqref{eq:LimsupInclusionLiouville} show that, even when the dynamics is not convex-valued, the infinitesimal behaviour of the set-valued Koopman operators is captured by the natural set-valued counterpart of the Liouville operators introduced in Definition \ref{eq:LiouvilleDef}. In Theorem \ref{thm:Eigenvalues} below, we shall see how this particular fact allows for the derivation of a set-valued spectral mapping theorem, that can be notably applied to finite control sets which are inherently non-convex. Secondly, the identity \eqref{eq:LiouvilleCharac} means in turn that, when the admissible velocities are convex in the sense of Hypothesis \textnormal{\ref{hyp:C}}, the set-valued Liouville operator is the infinitesimal generator of the Koopman semigroup. This transcribes the fact that even though time-dependent Koopman observables cannot be represented in general by a single controlled flow  -- see Figure \ref{fig:MeasurableSel} and Proposition \ref{prop:KoopmanRepresentation} above --, the set of all possible infinitesimal evolutions at a time $t$ close to $\tau$ coincides exactly with those generated by a controlled vector field $f_u \in \Fpazo$ with $u \in U$.
\end{remark}

We close this section by showing in the next proposition how the set-valued Liouville operator encodes the dynamics of those time-dependent Koopman observables which are generated by precisely one control signal, and thus follow an admissible flow of the system.

\begin{proposition}[Dynamics of time-dependent Koopman observables]
\label{prop:KoopmanDynamics}
Fix some $\varphi \in \Dpazo$, suppose that Hypotheses \textnormal{\ref{hyp:H}} hold and assume that $f_u \in C^1(\R^d,\R^d)$ for each $u \in U$. Then, a family of Koopman observables $(\tau,t) \in [0,T] \times [0,T] \mapsto \psi_{(\tau,t)} \in \Kpazo_{(\tau,t)}(\varphi)$ is of the form $\psi_{(\tau,t)} = \varphi \circ \Phi_{(\tau,t)}^u$ for some fixed control signal $u(\cdot) \in \Upazo$ \textnormal{if and only if} it is a strong solution of the differential inclusion
\begin{equation}
\label{eq:KoopmanDynamics}
\left\{
\begin{aligned}
& \partial_{\tau} \psi_{(\tau,t)} \in -\Lpazo(\psi_{(\tau,t)}), \\
& \psi{(t,t)} = \varphi, 
\end{aligned}
\right.
\end{equation}
in the space of observables $(\Dpazo,\Norm{\cdot}_{\Dpazo})$.
\end{proposition}

\begin{proof}
We start by assuming that $\psi_{(\tau,t)} = \varphi \circ \Phi_{(\tau,t)}^u \in \Dpazo$ for some $u(\cdot) \in \Upazo$, and note that, by invoking the results e.g. of \cite[Proposition A.6]{SemiSensitivity} which provide uniform-in-space linearisation results for characteristic flows associated with time-measurable dynamics, the strong limit of difference quotients
\begin{equation}
\label{eq:StrongDerivative}
\partial_{\tau} \psi_{(\tau,t)} := \lim_{h \to 0} \frac{\psi_{(\tau+h,t)} - \psi_{(\tau,t)}}{h} 
\end{equation}
is well-defined in $\Xpazo$ for $\Lcal^1$-almost every $\tau \in [0,T]$ and all times $t \in [0,T]$. Besides, one can easily check that for all $(t,y) \in [0,T] \times \R^d$, the map $\tau \in [0,T] \mapsto \psi_{(\tau,t)} \circ \Phi_{(t,\tau)}^u(y)$ is constant, so that
\begin{equation*}
\begin{aligned}
\tderv{}{\tau} \Big( \psi_{(\tau,t)} \circ \Phi_{(t,\tau)}^u(y) \Big) & = \partial_{\tau} \psi_{(\tau,t)} \circ \Phi_{(t,\tau)}^u(y) + \nabla_x \psi_{(\tau,t)} \circ \Phi_{(t,\tau)}^u(y) \cdot \partial_\tau \Phi_{(t,\tau)}^u(y) \\
& = \Big( \partial_{\tau} \psi_{(\tau,t)} + \nabla_x \psi_{(\tau,t)} \cdot f_{u(\tau)} \Big) \circ \Phi_{(t,\tau)}^u(y) = 0
\end{aligned}
\end{equation*}
for $\Lcal^1$-almost every $\tau \in [0,T]$. Noting that in the previous identity, one may choose $y = \Phi_{(\tau,t)}^u(x)$ for some arbitrary $x \in \R^d$, it then follows that $(\tau,x) \in [0,T] \times \R^d \mapsto \psi_{(\tau,t)}(x) \in \R^d$ is a classical solution of the Cauchy problem
\begin{equation}
\label{eq:TransportEquation}
\left\{
\begin{aligned}
& \partial_{\tau} \psi_{(\tau,t)} + \nabla_x \psi_{(\tau,t)} \cdot f_{u(\tau)} = 0, \\
& \psi_{(t,t)} = \varphi, 
\end{aligned}
\right.
\end{equation}
in $[0,T] \times \R^d$. At this stage, there remains to notice that the map 
\begin{equation*}
(\tau,x) \in [0,T] \times \R^d \mapsto \nabla_x \psi_{(\tau,t)}(x) \cdot f_{u(\tau)}(x) \in \R^d
\end{equation*}
is $\Lcal^1$-measurable in $\tau \in [0,T]$ as well as continuous in $x \in \R^d$, which implies that its functional lift $\tau \in [0,T] \mapsto \nabla_x \psi_{(\tau,t)} \cdot f_{u(\tau)} \in \Lpazo(\psi_{(\tau,t)})$ is $\Lcal^1$-measurable, see e.g. \cite[Page 511]{Papageorgiou1986}. The latter fact, together with \eqref{eq:StrongDerivative} and \eqref{eq:TransportEquation}, allow us to conclude that $\tau \in [0,T] \mapsto \psi_{(\tau,t)} \in \Dpazo$ is a strong solution of \eqref{eq:KoopmanDynamics}. 

To prove the converse implication, observe that one may rewrite the evaluation of the Liouville operator along a curve of observables $\tau \in [0,T] \mapsto \psi_{(\tau,t)} \in \Dpazo$ as 
\begin{equation*}
\Lpazo(\psi_{(\tau,t)}) = \Big\{ \Psi_t(\tau,u) ~\,\textnormal{s.t.}~ u \in U \Big\}, 
\end{equation*}
where $ \Psi_t : (\tau,u) \in [0,T] \times U \mapsto \nabla_x \psi_{(\tau,t)} \cdot f_u \in \Xpazo$. It is clear from Hypothesis \textnormal{\ref{hyp:H}} that the function $u \in U \mapsto \Psi_t(\tau,u) \in \Xpazo$ is continuous for $\Lcal^1$-almost every $\tau \in [0,T]$, and we just showed that $\tau \in [0,T] \mapsto \Psi_t(\tau,u) \in \Xpazo$ is $\Lcal^1$-measurable. Thus, because $\tau \in [0,T] \mapsto \partial_{\tau} \psi_{(\tau,t)} \in \Xpazo$ is an $\Lcal^1$-measurable map, and since the completion of $(\Xpazo,\Norm{\cdot}_{\Xpazo})$ is a separable Banach space, it follows from the measurable selection principle of Theorem \ref{thm:FilippovSel} that 
\begin{equation}
\label{eq:TransportEqInterm}
-\partial_{\tau}  \psi_{(\tau,t)} = \Psi_t(u_{\varphi}(\tau)) = \nabla_x \psi_{(\tau,t)} \cdot f_{u_{\varphi}(\tau)}
\end{equation}
for $\Lcal^1$-almost every $\tau \in [0,T]$ and some admissible control signal $u_{\varphi}(\cdot) \in \Upazo$. Owing to the regularity assumptions posited above, it follows from classical well-posedness results for transport equations (see e.g. \cite[Proposition 2.3]{AmbrosioC2014}) that 
\begin{equation*}
\psi_{(\tau,t)}(x) = \varphi \circ \Phi_{(\tau,t)}^{u_{\varphi}}(x)
\end{equation*}
for all $(\tau,x) \in [0,T] \times \R^d$, which concludes the proof of our claim.
\end{proof}

\begin{remark}[Comparing Proposition \ref{prop:KoopmanDynamics} with its classical counterpart]
\label{rmk:KoopmanDynamics}
In the familiar situation in which $(\Kpazo_t(\varphi))_{t \in [0,T]}$ is the usual Koopman semigroup generated by a single flow $(\Phi_t)_{t \in [0,T]} \subset C^0(\R^d,\R^d)$, it is well-known (see e.g. \cite[Section 7.6]{Lasota1998}) that the curve of observables $t \in [0,T] \mapsto \psi_t := \Kpazo_t(\varphi) \in \Xpazo$ is the unique strong solution of the Koopman dynamics 
\begin{equation}
\label{eq:ClassicalKoopmanDyn}
\left\{
\begin{aligned}
& \partial_t \psi_t = \nabla_x \psi_t \cdot f, \\
& \psi_0 = \varphi.
\end{aligned}
\right.
\end{equation} 
In the previous expression, the right-hand side coincides with the evaluation of the classical Liouville operator $\Lpazo : \varphi \in \Dpazo \mapsto \nabla_x \varphi \cdot f \in \Xpazo$ along $(\psi_t)_{t \in [0,T]}$. We claim that this result is essentially contained in Proposition \ref{prop:KoopmanDynamics}. Indeed, in the autonomous case the Koopman observables take the simpler form $\psi_{(\tau,t)} = \varphi \circ \Phi_{t-\tau}$, so that 
\begin{equation*}
\partial_{\tau} \psi_{(\tau,t)} = -\partial_t \psi_{(\tau,t)} \in -\Lpazo(\psi_{(\tau,t)}) = \big\{ -\nabla_x \psi_{(\tau,t)} \cdot f \, \big\}. 
\end{equation*}
Hence, up to redefining the time variable, the dynamics in \eqref{eq:KoopmanDynamics} reduces to \eqref{eq:ClassicalKoopmanDyn}.   
\end{remark}

\subsection{Set-valued Perron-Frobenius operators and Koopman adjoints}
\label{subsection:Perron}

In this section, we propose a formal definition for the set-valued Perron-Frobenius operator. We shall precisely discuss its relationship with the Koopman semigroup, and that of its infinitesimal generator with the Liouville operator. In this context, we will frequently resort to the notion of image measure under the action of a Borel map characterised in \eqref{eq:ImageMeasure} above.

\begin{definition}[Set-valued Perron-Frobenius operators]
\label{def:Perron}
We define the \textnormal{set-valued Perron-Frobenius operators} $\Ppazo_{(\tau,t)} : \Xpazo^* \tto \Xpazo^*$ by
\begin{equation}
\label{eq:PerronDef}
\Ppazo_{(\tau,t)}(\mu) := \Big\{ \Phi^u_{(\tau,t) \sharp \,} \mu ~\,\textnormal{s.t.}~ u(\cdot) \in \Upazo \Big\} \subset \Xpazo^*
\end{equation}
for all times $\tau,t \in [0,T]$ and each $\mu \in \Xpazo^*$. 
\end{definition}

In what follows, we provide mathematical grounding for this definition by showing that the set-valued Perron-Frobenius operators are related to suitable adjoints of the set-valued Koopman operators, inspired by the work \cite{Ioffe1981}.

\begin{definition}[Adjoints of closed processes]
\label{def:Adjoint}
Given two locally convex topological vector spaces $\Xpazo$ and $\Ypazo$, we define the \textnormal{adjoint} $\Fpazo^*: \Ypazo^* \tto \Xpazo^*$ of a closed process $\Fpazo : \Xpazo \tto \Ypazo$ by 
\begin{equation*}
\Fpazo^*(\mu) := \bigg\{ \nu \in \Xpazo^* ~\,\text{s.t.}~ \langle \nu , x \rangle_{\Xpazo} \leq \sup_{y \in \Fpazo(x)} \langle \mu , y \rangle_{\Ypazo} ~~ \text{for all $x \in \Xpazo$} \bigg\}
\end{equation*}
for all $\mu \in \Ypazo^*$.
\end{definition}

\begin{example}[Adjoints of fans generated by convex families of linear operators]
To see again why the proposed notion of adjoint is natural and compatible with its single-valued counterpart, suppose that $(\Xpazo,\Norm{\cdot}_{\Xpazo})$ and $(\Ypazo,\Norm{\cdot}_{\Ypazo})$ are Banach spaces and let $\Fpazo(x) = \{Ax\}$ for each $x \in \Xpazo$ and some bounded linear map $A : \Xpazo \to \Ypazo$. Then, one may trivially check that 
\begin{equation*}
\sup_{y \in \Fpazo(x)} \langle \mu , y \rangle_{\Ypazo} = \langle \mu , Ax\rangle_{\Ypazo} = \langle A^* \mu , x \rangle_{\Xpazo}
\end{equation*}
for each $\mu \in \Ypazo^*$ and every $x \in \Xpazo$, and it directly follows from Definition \ref{def:Adjoint} that
\begin{equation*}
\Fpazo^*(\mu) = \{A^* \mu \}.
\end{equation*}
More importantly, it is shown in \cite{Ioffe1981} that if $\Apazo$ is a convex and (weakly) closed collection of bounded linear operators from $\Xpazo$ into $\Ypazo$, then the adjoint of the set-valued mapping defined by $\Fpazo(x) := \{ Ax ~\,\textnormal{s.t.}~ A \in \Apazo \}$ for each $x \in \Xpazo$ writes as $\Fpazo^*(\mu) = \{A^* \mu ~\, \textnormal{s.t.}~ A \in \Apazo\}$ for all $\mu \in \Ypazo^*$.  
\end{example}

\begin{theorem}[Set-valued adjoints of the Liouville and Koopman operators]
\label{thm:Adjoints}
Suppose that Hypotheses \textnormal{\ref{hyp:H}} hold. Then, one has that
\begin{equation}
\label{eq:Adjoints}
\Lpazo^*(\mu) = \co \Big\{ \hspace{-0.1cm} -\Div_x (f_u \mu) ~\, \textnormal{s.t.}~ u \in U \Big\} \qquad \text{and} \qquad \Kpazo^*_{(\tau,t)}(\mu) = \co \Ppazo_{(\tau,t)}(\mu)
\end{equation}
for all $\mu \in \Xpazo^*$, where the convex hulls are taken with respect to the weak-$^*$ topology. In particular, if Hypothesis \textnormal{\ref{hyp:C}} holds, then 
\begin{equation*}
\Lpazo^*(\mu) = \Big\{ \hspace{-0.1cm} -\Div_x (f_u \mu) ~\, \textnormal{s.t.}~ u \in U \Big\} 
\end{equation*} 
for all $\mu \in \Xpazo^*$.
\end{theorem}

\begin{proof}
The proof of this result relies on a general characterisation of closed convex hulls in locally convex spaces in terms of support functions, which stems itself from Hahn-Banach's separation principle. First, we start by noting that given any $\varphi \in \Dpazo$ and each $\psi := \nabla_x \varphi \cdot f \in \Lpazo(\varphi)$ with $f \in \Fpazo$, there holds
\begin{equation*}
\begin{aligned}
\langle \mu , \psi \rangle_{\Xpazo} & = \langle \mu , \nabla_x \varphi \cdot f \rangle_{\Xpazo} \\
& = \INTDom{\nabla_x \varphi(x) \cdot f(x)}{\R^d}{\mu(x)} = \langle - \Div_x(f \mu) , \varphi \rangle_{\Dpazo}
\end{aligned}
\end{equation*}
for all $\mu \in \Xpazo^*$, by definition the \eqref{eq:DivergenceDef} of the divergence, which implies that
\begin{equation*}
\begin{aligned}
\Lpazo^*(\mu) & = \bigg\{ \nu \in \Dpazo^* ~\,\text{s.t.}~ \langle \nu , \varphi \rangle_{\Dpazo} \leq \sup_{\psi \in \Lpazo(\varphi)} \langle \mu , \psi \rangle_{\Xpazo} ~~ \text{for all $\varphi \in \Dpazo$} \bigg\} \\
& = \bigg\{ \nu \in \Dpazo^* ~\,\text{s.t.}~ \langle \nu , \varphi \rangle_{\Dpazo} \leq \sup_{f \in \Fpazo} \langle - \Div_x(f \mu) , \varphi \rangle_{\Dpazo} ~~ \text{for all $\varphi \in \Dpazo$} \bigg\}.
\end{aligned}
\end{equation*}
At this stage, upon recalling that $\Dpazo^*$ endowed with the weak-$^*$ topology is a Hausdorff locally convex topological vector space (see e.g. \cite[Proposition 3.11]{Brezis}) whose topological dual is $\Dpazo$ itself (see e.g. \cite[Proposition 3.14]{Brezis}), it follows from the standard characterisation of closed convex hulls provided e.g. in \cite[Theorem 2.4.2]{Aubin1990} that 
\begin{equation*}
\Lpazo^*(\mu) = \co \Big\{ \hspace{-0.1cm} -\Div_x (f_u \mu) ~\, \textnormal{s.t.}~ u \in U \Big\}.
\end{equation*}
We now turn our attention to the adjoints of the Koopman operators. Similarly to what precedes, given $\varphi \in \Xpazo$ and $\psi_{(\tau,t)} = \varphi \circ \Phi_{(\tau,t)}^u \in \Kpazo_{(\tau,t)}(\varphi)$  with $u(\cdot) \in \Upazo$, we start by observing that
\begin{equation*}
\begin{aligned}
\langle \mu , \psi_{(\tau,t)} \rangle_{\Xpazo} & = \INTDom{\varphi \circ \Phi_{(\tau,t)}^u(x)}{\R^d}{\mu(x)} \\
& = \INTDom{\varphi(x)}{\R^d}{(\Phi_{(\tau,t) \sharp \,}^u \mu)(x)} = \langle \Phi_{(\tau,t) \sharp \,}^u \mu , \varphi \rangle_{\Xpazo}
\end{aligned}
\end{equation*}
for all $\mu \in \Xpazo^*$. Then, following the same reasoning as above, one may verify that 
\begin{equation*}
\begin{aligned}
\Kpazo^*_{(\tau,t)}(\mu) & = \bigg\{ \nu_{(\tau,t)} \in \Xpazo^* ~\,\textnormal{s.t.}~ \langle \nu_{(\tau,t)}, \varphi \rangle_{\Xpazo} \leq \sup_{\psi_{(\tau,t)} \in \Kpazo_{(\tau,t)}(\varphi)} \langle \mu , \psi_{(\tau,t)} \rangle_{\Xpazo} ~~ \text{for all $\varphi \in \Xpazo$} \bigg\} \\
& = \bigg\{ \nu_{(\tau,t)} \in \Xpazo^* ~\,\textnormal{s.t.}~ \langle \nu_{(\tau,t)}, \varphi \rangle_{\Xpazo} \leq \sup_{u(\cdot) \in \Upazo} \langle \Phi_{(\tau,t)\sharp \,}^u \mu , \varphi \rangle_{\Xpazo} ~~ \text{for all $\varphi \in \Xpazo$} \bigg\} = \co \Ppazo_{(\tau,t)}(\mu)
\end{aligned}
\end{equation*}
by Hahn-Banach's separation principle applied to $\Xpazo^*$ endowed with the weak-$^*$ topology. Lastly, if Hypothesis \textnormal{\ref{hyp:C}} holds, it may be easily verified that 
\begin{equation*}
\begin{aligned}
\Lpazo^*(\mu) & = \co \Big\{ \hspace{-0.1cm} -\Div_x (f \mu) ~\, \textnormal{s.t.}~ f \in \Fpazo \Big\} \\
& = \Big\{ \hspace{-0.1cm} -\Div_x (f \mu) ~\, \textnormal{s.t.}~ f \in \co \Fpazo \Big\} = \Big\{ \hspace{-0.1cm} -\Div_x (f_u \mu) ~\, \textnormal{s.t.}~ u \in U \Big\}
\end{aligned}
\end{equation*}
thanks to the compactness of $\co \Fpazo$ in the topology of local uniform convergence. 
\end{proof}

We end this section by showing that, under our working assumptions, the adjoint of the Liouville operator is the generator of the Perron-Frobenius semigroup\footnote{Therein, we make a small abuse of notation by taking Kuratowski-Painlevé limits in the weak-$^*$ topology although the latter is not metrisable. This could be made rigorous by redefining the notion in terms of neighbourhoods.}

\begin{theorem}[Generator of the Perron-Frobenius semigroup]
\label{thm:InfinitesimalPerron}
Suppose that Hypotheses \textnormal{\ref{hyp:H}} hold. Then, the following inclusions
\begin{equation}
\label{eq:LiminfInclusionPerron}
\Big\{ \hspace{-0.1cm} -\Div_x (f_u \mu) ~\, \textnormal{s.t.}~ u \in U \Big\} \subset \Liminf{t \to \tau} \, \frac{\Ppazo_{(\tau,t)}(\mu)-\mu}{t-\tau}
\end{equation}
and 
\begin{equation}
\label{eq:LimsupInclusionPerron}
\Limsup{t \to \tau} \, \frac{\Ppazo_{(\tau,t)}(\mu)-\mu}{t-\tau} \subset \co \Big\{ \hspace{-0.1cm} -\Div_x (f_u \mu) ~\, \textnormal{s.t.}~ u \in U \Big\}
\end{equation}
are satisfied for each $\mu \in \Xpazo^*$. Furthermore, if Hypothesis \textnormal{\ref{hyp:C}} holds, then the adjoint of the set-valued Liouville operator is the infinitesimal generator of the Perron-Frobenius semigroup, in the sense that
\begin{equation}
\label{eq:LimPerron}
\Lim{t \to \tau} \frac{\Ppazo_{(\tau,t)}(\mu) - \mu}{t-\tau} = \Lpazo^*(\mu)
\end{equation}
for all $\mu \in \Xpazo^*$, with the limit being taken in the weak-* topology. 
\end{theorem}

\begin{proof}
By Theorem \ref{thm:Adjoints}, we start by noting that if Hypothesis \textnormal{\ref{hyp:C}} holds, then 
\begin{equation*}
\Lpazo^*(\mu) = \Big\{ -\Div_x(f_u \mu) ~\,\textnormal{s.t.}~ u \in U \Big\} 
\end{equation*}
and \eqref{eq:LimPerron} directly stems from the inclusions \eqref{eq:LiminfInclusionPerron} and \eqref{eq:LimsupInclusionPerron}, namely the adjoint of the Liouville operator is then the infinitesimal generator of the Perron-Frobenius semigroup. We then seek to establish the liminf inclusion \eqref{eq:LiminfInclusionPerron}, which amounts to showing that given $\nu := -\Div_x(f_{\bar{u}} \mu) \in \Dpazo^*$, there exists some $u(\cdot) \in \Upazo$ for which  
\begin{equation*}
\lim_{t \to \tau} \frac{\Phi_{(\tau,t) \sharp \,}^u \mu -\mu}{t-\tau} = - \Div_x(f_{\bar{u}} \mu)
\end{equation*}
in the weak$^*$-topology. To this end, we simply consider the constant control signal given by $u(t) := \bar{u}$ for all times $t \in [0,T]$, and observe that by Lebesgue's dominated convergence theorem, it holds that
\begin{equation*}
\begin{aligned}
\lim_{t \to \tau} \big\langle \tfrac{1}{t-\tau} \big( \Phi_{(\tau,t) \sharp \,}^{\bar{u}} \mu -\mu \big) , \zeta \big\rangle_{\Xpazo} & = \lim_{t \to \tau} \INTDom{\tfrac{1}{t-\tau} \Big( \zeta \circ \Phi_{(\tau,t)}^{\bar{u}}(x) - \zeta(x) \Big)}{\R^d}{\mu(x)} \\
& = \INTDom{\nabla_x \zeta(x) \cdot f_{\bar{u}}(x) \,}{\R^d}{\mu(x)} = \big\langle -\Div_x(f_{\bar{u}} \, \mu) , \zeta \big\rangle_{\Dpazo}
\end{aligned}
\end{equation*}
for any $\zeta \in \Dpazo = C^1_c(\R^d,\C)$, thereby yielding the liminf inclusion. To prove the limsup inclusion \eqref{eq:LimsupInclusionPerron}, one needs to show that for every pair of sequences $(t_n) \subset [0,T] \setminus \{\tau\}$ and $(u_n(\cdot)) \subset \Upazo$, there exists a subsequence $t_{n_k} \to\, \tau$ such that 
\begin{equation*}
\lim_{t_{n_k} \to\, \tau} \frac{\Phi_{(\tau,t_{n_k}) \sharp \,}^{u_{n_k}} \mu - \mu}{t_{n_k}-\tau} \in \co \Lpazo^*(\mu)
\end{equation*}
in the weak-$^*$ topology. Given some $\zeta \in \Dpazo$, we consider the compact set 
\begin{equation*}
K_{\zeta} := \overline{\bigcup_{u(\cdot) \in \Upazo}\bigcup_{\tau,t \in [0,T]}  \Phi^u_{(t,\tau)}(\supp(\zeta))}
\end{equation*}
and note by reproducing the reasoning detailed in the proof of Theorem \ref{thm:RepresentationLiouville} above that there exists some $\bar{u} \in U$ such that 
\begin{equation*}
\sup_{x \in K_{\zeta}} \big| \Phi^{u_{n_k}}_{(\tau,t_{n_k})}(x) - x - (t_{n_k}-\tau) f_{\bar{u}}(x) \big| \leq o(|t_{n_k}-\tau|)
\end{equation*}
along a subsequence $t_{n_k} \to\, \tau$. The latter identity then yields
\begin{equation*}
\begin{aligned}
\lim_{t_{n_k} \to\, \tau} \big\langle \tfrac{1}{t_{n_k}-\, \tau} \big( \Phi^{u_{n_k}}_{(\tau,t_{n_k}) \sharp \,} \mu - \mu \big) , \zeta \big\rangle_{\Xpazo} & = \lim_{t_{n_k} \to \tau} \INTDom{\tfrac{1}{t_{n_k}-\,\tau} \Big( \zeta \circ \Phi^{u_{n_k}}_{(\tau,t_{n_k})}(x) - \zeta(x) \Big)}{\R^d}{\mu(x)} \\
& = \INTDom{\nabla_x \zeta(x) \cdot f_{\bar{u}}(x)}{\R^d}{\mu(x)} = \big\langle - \Div_x(f_{\bar{u}} \, \mu) , \zeta \big\rangle_{\Dpazo}
\end{aligned}
\end{equation*}
by Lebesgue's dominated convergence theorem along with the definition \eqref{eq:DivergenceDef} of the divergence distribution and the fact that $\zeta \in C^1_c(\R^d,\C)$.
\end{proof}

As a consequence of Theorem \ref{thm:InfinitesimalPerron}, it can be shown\footnote{The precise argument involves abstract measurable selection theorems adapted to weak-$^*$ topologies \iffalse -- which are not metrisable but Souslin  \cite[Theorem 7 page 112]{Schwartz1973} --\fi that may be found e.g. in \cite{Castaing1977}, and far exceeds the scope of this article.} by adapting the arguments of Proposition \ref{prop:KoopmanDynamics} that a curve $t \in [0,T] \mapsto \mu_{(\tau,t)} \in \Ppazo_{(\tau,t)}(\mu)$ is of the form $\mu_{(\tau,t)} = \Phi_{(\tau,t) \sharp \,}^u \mu$ for some fixed control signal $u(\cdot) \in \Upazo$ if and only if it is a weak-$^*$ solution of 
\begin{equation*}
\left\{
\begin{aligned}
& \partial_t \mu_{(\tau,t)} \in \Big\{ \hspace{-0.1cm} - \Div_x(f_u \mu_{(\tau,t)}) ~\, \textnormal{s.t.}~ u \in U \Big\} , \\
& \mu_{(\tau,\tau)} = \mu.
\end{aligned}
\right.
\end{equation*}
Interestingly, it follows from the characterisation derived in Theorem \ref{thm:Adjoints} that, when the velocities are convex, the adjoint Koopman dynamics is essentially a differential inclusion in the space of measures as introduced by the first author in \cite{ContInc,ContIncPp}. 

\subsection{Point spectra of the set-valued Koopman and Liouville operators}
\label{subsection:Spectral}

In this section, we briefly study the interplay between the point spectra of the set-valued Koopman and Liouville operators, understood in the following sense.

\begin{definition}[Eigenvalues and eigenvectors of closed processes]
\label{def:Spectrum}
Given two Banach spaces $(\Xpazo,\Norm{\cdot}_{\Xpazo})$ and $(\Ypazo,\Norm{\cdot}_{\Ypazo})$, we say that $\lambda \in \C$ is an \textnormal{eigenvalue} of a closed process $\Fpazo : \Xpazo \tto \Ypazo$ if there exists an \textnormal{eigenvector} $x_{\lambda} \in \Xpazo \setminus \{ 0\}$ such that
\begin{equation*}
\lambda x_{\lambda} \in \Fpazo(x_{\lambda}). 
\end{equation*}
The set eigenvalues of $\Fpazo$ is called the \textnormal{point spectrum} and denoted by $\sigma_p(\Fpazo) \subset \C$. 
\end{definition}

\begin{example}[Link with the usual notion of point spectrum]
Again, one may easily check that the definition of spectrum introduced above reduces to the usual one when $\Fpazo(x) = \{A x \}$ with $A : \Xpazo \to \Ypazo$ linear and bounded, since then $\lambda x_{\lambda} \in \Fpazo(x_{\lambda}) = \{ A x_{\lambda}\}$ if and only if $A x_{\lambda} = \lambda x_{\lambda}$ whenever $\lambda \in \sigma_p(\Fpazo)$.
\end{example}

Below, we prove a set-valued version of the classical spectral mapping theorem (see e.g. \cite[Chapter IV - Theorem 3.7]{Engel2001}), which relates the point spectra of the Liouville and Koopman operators.  

\begin{theorem}[Set-valued spectral mapping theorem]
\label{thm:Eigenvalues}
Suppose that Hypotheses \textnormal{\ref{hyp:H}} hold and fix some $\tau \in [0,T]$. Then for all times $t \in [\tau,T]$, one has that
\begin{equation}
\label{eq:SpectralInc1}
e^{\sigma_p(\Lpazo) (t-\tau)} \subset \sigma_p(\Kpazo_{(\tau,t)}), 
\end{equation}
namely if an observable $\varphi_{\lambda} \in\Dpazo$ is an eigenfunction of $\Lpazo : \Dpazo \tto \Xpazo$ with eigenvalue $\lambda \in \sigma_p(\Lpazo)$, then it is an eigenfunction of $\Kpazo_{(\tau,t)} : \Xpazo \tto \Xpazo$ with eigenvalue $e^{\lambda (t-\tau)} \in \sigma_p(\Kpazo_{(\tau,t)})$. Furthermore, if Hypothesis \textnormal{\ref{hyp:C}} holds, then for each $\Lcal^1$-measurable curve $t \in [0,T] \mapsto \lambda_{(\tau,t)} \in \C$ satisfying
\begin{equation*}
\lambda_{(\tau,t)} \varphi_{\lambda} \in \Kpazo_{(\tau,t)}(\varphi_{\lambda})
\end{equation*}
for some given time-independent $\varphi_{\lambda} \in \Dpazo$, there exists a sequence $t_n \to \tau$ for which
\begin{equation}
\label{eq:SpectralInc2First}
\lim_{n \to +\infty} \lambda_{(\tau,t_n)} = 1 \qquad \text{and} \qquad \lim_{n \to +\infty} \frac{\lambda_{(\tau,t_n)} - 1}{t_n-\tau} \in \sigma_p(\Lpazo).
\end{equation}
In particular, the converse spectral inclusion 
\begin{equation}
\label{eq:SpectralInc2}
\begin{aligned}
\bigg\{ e^{\lambda (t-\tau)} \in \C ~\, \textnormal{s.t.}~ e^{\lambda (t-\tau)} \varphi_{\lambda} \in \Kpazo_{(\tau,t)}(\varphi_{\lambda}) ~ \text{with}~ \lambda \in \C \bigg\} \subset e^{\sigma_p(\Lpazo) (t-\tau)}
\end{aligned}
\end{equation}
holds for all times $\tau,t \in [0,T]$.
\end{theorem}

\begin{proof}
We start by proving the first inclusion \eqref{eq:SpectralInc1}. By Theorem \ref{thm:RepresentationLiouville}, the fact that $\lambda \varphi_{\lambda} \in \Lpazo(\varphi_{\lambda})$ for some $\lambda \in\sigma_p(\Lpazo)$ and $\varphi_{\lambda} \in\Dpazo$ is tantamount to the existence of some control value $\bar{u}_{\lambda} \in U$ for which
\begin{equation}
\lambda \varphi_{\lambda} = \nabla_x\varphi_{\lambda} \cdot f_{\bar{u}_{\lambda}}.
\end{equation}
Hence, considering the constant control signal defined by $\bar{u}_{\lambda}(t) := \bar{u}_{\lambda}$ for all times $t \in[0,T]$, one has 
\begin{equation*}
\tderv{}{t}{} \Big( \varphi \circ \Phi^{u_{\lambda}}_{(\tau,t)}(x) \Big) = \langle \nabla_x \varphi_{\lambda} , f_{\bar{u}_{\lambda}} \rangle \circ \Phi_{(\tau,t)}^{u_{\lambda}}(x) = \lambda \varphi_{\lambda} \circ \Phi_{(\tau,t)}^{u_{\lambda}}(x)
\end{equation*}
for each $x \in\R^d$, and it stems from standard Cauchy-Lipschitz uniqueness that
\begin{equation*}
\varphi \circ \Phi^{u_{\lambda}}_{(\tau,t)}(x) = e^{\lambda (t-\tau)} \varphi_{\lambda}(x)
\end{equation*}
for all $(t,x) \in [0,T] \times\R^d$. The thesis follows then from the fact that $\varphi \circ \Phi^{u_{\lambda}}_{(\tau,t)} \in\Kpazo_{(\tau,t)}(\varphi_{\lambda})$, by the very definition of the set-valued Koopman operators. 

Conversely, let us assume that there exist an $\Lcal^1$-measurable map $t \mapsto\lambda_{(\tau,t)} \in \C$ and $\varphi_{\lambda} \in\Dpazo$ such that $\lambda_{(\tau,t)} \varphi_{\lambda} \in\Kpazo_{(\tau,t)}(\varphi_{\lambda})$ for $\Lcal^1$-almost every $t \in[\tau,T]$, and suppose that Hypothesis \textnormal{\ref{hyp:C}} holds. By Proposition \ref{prop:KoopmanRepresentation}, there exists an $\Lcal^2$-measurable map $(\tau,t) \in [0,T] \times [0,T] \mapsto u_{(\tau,t)}^{\lambda}(\cdot) \in\Upazo$ such that
\begin{equation}
\label{eq:RepresentationLambdaKoopman}
\lambda_{(\tau,t)} \varphi_{\lambda} = \varphi_{\lambda} \circ \Phi_{(\tau,t)}^{u_{(\tau,t)}^{\lambda}} \in \Kpazo_{(\tau,t)}(\varphi_{\lambda})
\end{equation}
for $\Lcal^1$-almost every $t \in [\tau,T]$. By repeating the arguments yielding the limsup inclusion of Theorem \ref{thm:RepresentationLiouville}, one may infer the existence of an element $\bar{u}_{\lambda} \subset U$ along with a sequence $t_n \to \tau^+$ such that 
\begin{equation*}
\varphi_{\lambda} \circ \Phi_{(\tau,t_n)}^{u_{(\tau,t_n)}^{\lambda}} = \varphi_{\lambda} + (t_n-\tau) \nabla_x \varphi_{\lambda} \cdot f_{\bar{u}_{\lambda}} + o(|t_n-\tau|).
\end{equation*}
This along with \eqref{eq:RepresentationLambdaKoopman} implies that $\lim_{n \to +\infty} \lambda_{(\tau,t_n)} = 1$, as well as
\begin{equation*}
\begin{aligned}
\Big( \lim_{n \to +\infty} \frac{\lambda_{(\tau,t_n)}-1}{t_n-\tau} \Big) \varphi_{\lambda} & = \lim_{n \to +\infty} \frac{\varphi_{\lambda} \circ \Phi_{(\tau,t_n)}^{u_{(\tau,t_n)}^{\lambda}} - \varphi_{\lambda}}{t_n-\tau} = \nabla_x \varphi_{\lambda} \cdot f_{\bar{u}_{\lambda}}, 
\end{aligned}
\end{equation*}
which belongs to $\Lpazo(\varphi_{\lambda})$ as a consequence of Theorem \ref{thm:RepresentationLiouville}. Thus, we have shown the limit inclusion \eqref{eq:SpectralInc2First}, from which the spectral inclusion \eqref{eq:SpectralInc2} easily follows.
\end{proof}

\begin{example}[The spectral mapping theorem for linear feedback controls]
To illustrate the previous theorem, let us focus on the simple example of a linear time-invariant system with constant feedback controls 
\begin{equation*}
\dot x(t) = (A + BK)x(t) 
\end{equation*}
wherein $(A,B) \in \R^{d \times d} \times \R^{d \times m}$ and $K \in \Kpazo_{\adm}$ with $\Kpazo_{\adm} \subset \R^{m \times d}$ being a compact set. In practice, one may consider for instance that $\Kpazo_{\adm} := \{K_1,\dots,K_n\}$ is comprised of finitely many feedback matrices, or maybe fix $\Kpazo_{\adm} := \big\{ K \in \R^{m \times d} ~\, \textnormal{s.t.} \; \Norm{K}_{\mathbb{F}} \, \leq 1 \big\}$ as the closed unit ball for the Frobenius norm. In this context, the set-valued Koopman and Liouville operators can be computed explicitly as 
\begin{equation*}
\Kpazo_{(\tau,t)} : \varphi \in \Xpazo \tto \bigg\{ x \in \R^d \mapsto \varphi \bigg( \exp \Big((t-\tau) (A+BK) \Big)x \bigg) \in \C ~\, \textnormal{s.t.}~ K \in \Kpazo_{\adm} \bigg\} \subset \Xpazo
\end{equation*}
and 
\begin{equation*}
\Lpazo : \varphi \in \Dpazo \tto \bigg\{ x \in \R^d \mapsto \nabla_x \varphi(x) \cdot \Big((A+BK)x \Big) \in \C ~\, \textnormal{s.t.}~ K \in \Kpazo_{\adm} \bigg\} \subset \Xpazo.
\end{equation*}
Then, a map $\varphi_{\lambda} \in \Dpazo$ is an eigenvalue of the Liouville operator if there exists some $\lambda \in \C$ along with a matrix $K_{\lambda} \in \Kpazo_{\adm}$ such that 
\begin{equation*}
\lambda \varphi_{\lambda}(x) = \nabla_x \varphi_{\lambda}(x) \cdot (A+BK_{\lambda})x 
\end{equation*}
for all $x \in \R^d$. Following e.g. \cite{Mauroy2016}, one may then consider candidates linear eigenfunctions of the form $\varphi_{\lambda}(x) := e_{\lambda} \cdot x$, with $\lambda \in \C$ and $e_{\lambda} \in \C^d$ being respectively an eigenvalue and eigenvector of $(A+BK_{\lambda})^* \in \C^{d \times d}$, and check that $e^{\lambda t} \in \sigma_p(\Kpazo_{(0,t)})$ for each $\lambda \in \sigma_p(\Lpazo)$, and vice versa.
\end{example}

We end this section by a simple proposition which shows how one may produce new eigenfunctions and eigenvalues of the set-valued Liouville operator by combining adequate subfamilies of already known ones.  

\begin{proposition}[On the structure of the point spectrum]
Let $\varphi_{\lambda_1},\varphi_{\lambda_2} \in\Dpazo$ be two eigenfunctions of $\Lpazo : \Dpazo \tto \Xpazo$ associated with the same control $u_{\lambda} \in U$, namely 
\begin{equation*}
\lambda_1 \varphi_{\lambda_1} = \nabla_x \varphi_{\lambda_1} \cdot f_{u_{\lambda}} \qquad \text{and} \qquad \lambda_2 \varphi_{\lambda_2} = \nabla_x \varphi_{\lambda_2} \cdot f_{u_{\lambda}}.     
\end{equation*}
Then for every $\alpha_1,\alpha_2 \in \R$ such that $\varphi_{\lambda_1}^{\alpha_1} \varphi_{\lambda_2}^{\alpha_2} \in \Dpazo$, the latter is an eigenfunction of $\Lpazo : \Dpazo \tto \Xpazo$ with the eigenvalue $(\alpha_1 \lambda_1 + \alpha_2 \lambda_2) \in \sigma_p(\Lpazo)$. 
\end{proposition}

\begin{proof}
Given $\alpha_1,\alpha_2 \in \R$ for which $\varphi_{\lambda_1}^{\alpha_1} \varphi_{\lambda_2}^{\alpha_2} \in \Dpazo$, one easily gets that 
\begin{equation*}
\begin{aligned}
\nabla_x (\varphi_{\lambda_1}^{\alpha_1} \varphi_{\lambda_2}^{\alpha_2}) \cdot f_{u_{\lambda}} & = \alpha_1 \varphi_{\lambda_1}^{\alpha_1-1} \varphi_{\lambda_2}^{\alpha_2} \; \nabla_x \varphi_{\lambda_1} \cdot f_{u_{\lambda}} + \alpha_2 \varphi_{\lambda_1}^{\alpha_1} \varphi_{\lambda_2}^{\alpha_2-1} \; \nabla_x \varphi_{\lambda_2} \cdot f_{u_{\lambda}} \\
& = (\alpha_1 \lambda_1 + \alpha_2 \lambda_2) \varphi_{\lambda_1}^{\alpha_1} \varphi_{\lambda_2}^{\alpha_2}, 
\end{aligned}
\end{equation*}
which proves the statement. 
\end{proof}



\section{Conclusion}

In this paper, we proposed a novel mathematical framework for Koopman operators associated with control systems, formulated using the tools and concepts of set-valued analysis. The main rationale behind this work was to provide a sound Koopman theory for systems with inputs, without assuming any dynamical evolution on the latter. The next natural step in this line of research is to develop numerical methods based on these theoretical foundations, e.g. to develop a suitable set-valued Extended Dynamic Mode Decomposition. 


\addcontentsline{toc}{section}{Appendices}
\section*{Appendices}


\setcounter{section}{0} 
\renewcommand{\thesection}{A} 
\renewcommand{\thesubsection}{A} 

\subsection{Proof of Corollary \ref{cor:ContinuousFlow}}
\label{section:AppendixContinuity}

\setcounter{equation}{0} \renewcommand{\theequation}{A.\arabic{equation}}

In this appendix section, we detail the proof of Corollary \ref{cor:ContinuousFlow} for the sake of self-containedness.

\begin{proof}[Proof of Corollary \ref{cor:ContinuousFlow}]Consider an element $u(\cdot) \in \Upazo$ as well as a sequence $(u_n(\cdot)) \subset \Upazo$ such that 
\begin{equation*}
\INTSeg{\dsf_U(u_n(t),u(t))}{t}{0}{T} ~\underset{n \to +\infty}{\longrightarrow}~ 0. 
\end{equation*}
This implies that $(u_n(\cdot))$ converges to $u(\cdot)$ in measures (see e.g. \cite[Remark 1.18]{AmbrosioFuscoPallara}), i.e. for each $\delta > 0$ the sets $\Apazo_n^{\delta} := \big\{ t \in [0,T] ~\, \textnormal{s.t.}~ \dsf_U(u_n(t),u(t)) \geq \delta \big\}$ are such that
\begin{equation}
\label{eq:ConvergenceMeasures}
\Lcal^1 (\Apazo_n^{\delta}) ~\underset{n \to +\infty}{\longrightarrow}~ 0.
\end{equation}
In turn, given a compact set $K \subset \R^d$, it follows from \eqref{eq:StabEst} in Theorem \ref{thm:Well-posed} that 
\begin{equation*}
\sup_{n \geq 1} \big| \Phi_{(\tau,t)}^{u_n}(x) \big| \leq R_K 
\end{equation*}
for all $(\tau,t,x) \in [0,T] \times [0,T] \times K$ and some constant $R_K > 0$ which only depends on the magnitudes of $m,T$ and $\sup_{x \in K}|x|$. Then, introducing the notation $K' := B(0,R_K)$, one can use Hypothesis \textnormal{\ref{hyp:H}}-$(ii)$ along with Gr\"onwall's lemma to estimate the discrepancy between the flow maps generated by $u_n(\cdot)$ and $u(\cdot)$ as
\begin{equation}
\label{eq:FlowGronwallEst}
\big| \Phi_{(\tau,t)}^u(x) - \Phi_{(\tau,t)}^{u_n}(x) \big| \leq \bigg( \sup_{x \in K} \INTSeg{\Big| f \Big( \Phi_{(\tau,s)}^u(x) , u(s) \Big) - f \Big( \Phi_{(\tau,s)}^u(x) , u_n(s) \Big) \Big|}{s}{0}{T} \bigg) e^{\ell_{K'} T} 
\end{equation}
for all $(\tau,t,x) \in [0,T] \times [0,T] \times K$ and each $n \geq 1$. At this stage, given some arbitrary $\epsilon > 0$, it follows from Hypotheses \textnormal{\ref{hyp:H}} along with the fact that $[0,T] \times [0,T] \times K \times U$ is a compact metric space that there exists some $\delta >0$ for which
\begin{equation}
\label{eq:FLowContEst1}
\sup_{(\tau,t,x) \in [0,T] \times [0,T] \times K} \Big| f \Big( \Phi_{(\tau,t)}^u(x) , u_1 \Big) - f \Big( \Phi_{(\tau,t)}^u(x) , u_2 \Big) \Big| < \frac{\epsilon}{2T e^{\ell_{K'}T}}
\end{equation}
whenever $u_1,u_2 \in U$ are such that $\dsf_U(u_1,u_2) < \delta$. Moreover, following \eqref{eq:ConvergenceMeasures}, there exists for that same $\delta >0$ some integer $N_{\epsilon,K} \geq 1$ such that 
\begin{equation}
\label{eq:FLowContEst2}
\Lcal^1(\Apazo_n^{\delta}) < \frac{\epsilon}{4m(1+R_K) e^{\ell_{K'}T}}
\end{equation} 
for each $n \geq N_{\epsilon}$. Therefore, by combining \eqref{eq:FLowContEst1} and \eqref{eq:FLowContEst2}, one may infer that
\begin{equation*}
\begin{aligned}
& \INTSeg{\Big| f \Big( \Phi_{(\tau,s)}^u(x) , u(s) \Big) - f \Big( \Phi_{(\tau,s)}^u(x) , u_n(s) \Big) \Big|}{s}{0}{T} \\
&\leq \INTDom{\Big| f \Big( \Phi_{(\tau,s)}^u(x) , u(s) \Big) - f \Big( \Phi_{(\tau,s)}^u(x) , u_n(s) \Big) \Big|}{\Apazo_n^{\delta}}{s} \\
& \hspace{0.4cm} + \INTDom{\Big| f \Big( \Phi_{(\tau,s)}^u(x) , u(s) \Big) - f \Big( \Phi_{(\tau,s)}^u(x) , u_n(s) \Big) \Big|}{[0,T] \backslash \Apazo_n^{\delta}}{s} \\
& \leq 2m( 1 + R_K ) \Lcal^1(\Apazo_n^{\delta}) \\
& \hspace{0.4cm} + \sup_{(\tau,s,x) \in [0,T] \times [0,T] \times K} \Big| f \Big( \Phi_{(\tau,s)}^u(x) , u(s) \Big) - f \Big( \Phi_{(\tau,s)}^u(x) , u_n(s) \Big) \Big| \, \Lcal^1([0,T] \setminus \Apazo_n^{\delta}) \\
& \leq \frac{\epsilon}{e^{\ell_{K'} T}}
\end{aligned}
\end{equation*}
for all $x \in K$ and each $n \geq N_{\epsilon,K}$, where we used Hypothesis \textnormal{\ref{hyp:H}}-$(i)$. Plugging this last inequality in \eqref{eq:FlowGronwallEst} finally yields the uniform estimate   
\begin{equation*}
\sup_{(\tau,t,x) \in [0,T] \times [0,T] \times K} \big| \Phi_{(\tau,t)}^u(x) - \Phi_{(\tau,t)}^{u_n}(x) \big| \leq \epsilon
\end{equation*}
for each $n \geq N_{\epsilon,K}$, which concludes the proof since $K \subset \R^d$ and $\epsilon > 0$ are arbitrary.
\end{proof}


\section*{Acknowledgements}

{\small
This research was part of the programme DesCartes and is supported by the National Research Foundation, Prime Minister's Office, Singapore under its Campus for Research Excellence and Technological Enterprise (CREATE) programme. This work was also co-funded by the European Union under the project ROBOPROX (reg.~no.~CZ.02.01.01/00/22\_008/0004590). The authors thank the anonymous referees for their careful reading and comments which greatly contributed to improve the manuscript.}


\bibliographystyle{plain}
{\footnotesize
\bibliography{References}
}

\end{document}